\documentclass[reqno]{amsart}
\usepackage{amsmath, amssymb, amsthm, amsrefs, amscd, amsfonts, mathtools}
\usepackage[all]{xy}

\usepackage{xcolor}
\usepackage[english]{babel}
\usepackage[utf8]{inputenc}	

\newtheorem{thm}{Theorem}[section]

\newtheorem{prop}[thm]{Proposition}
\newtheorem{theorem}[thm]{Theorem}
\newtheorem{lem}[thm]{Lemma}
\newtheorem{lema}[thm]{Lemma}
\newtheorem{cor}[thm]{Corollary}
\theoremstyle{definition}

\newtheorem{definition}[thm]{Definition}
\newtheorem{rem}[thm]{Remark}

\newtheorem{exa}[thm]{Example}

\usepackage{hyperref}
\numberwithin{equation}{section}\theoremstyle{plain}

\newcommand{\B}{\mathcal{B}}

\theoremstyle{definition}

\def \Z {\mathbb{Z}}
\def \o {\otimes}
\def \k {\Bbbk}
\def \aa {\alpha}

\allowdisplaybreaks

 \title[Partial actions of a Hopf algebra on its base field]{Partial actions of a Hopf algebra on its base field and the corresponding partial smash product algebra}
\author[Martini, Paques and Silva]{Grasiela Martini, Antonio Paques and Leonardo Duarte Silva}

\address[Martini]{Universidade Federal do Rio Grande, Brazil}
\email{grasiela.martini@furg.br}

\address[Paques]{Universidade Federal do Rio Grande do Sul, Brazil}
\email{paques@mat.ufrgs.br}

\address[Silva]{Universidade Federal do Rio Grande do Sul, Brazil}
\email{leonardoufpel@gmail.com}

\begin{document}

\allowdisplaybreaks

\begin{abstract}
	We introduce the concept of a $\lambda$-Hopf algebra as a Hopf algebra obtained as the partial smash product algebra of a Hopf algebra and its base field, and show that every Hopf algebra is a $\lambda$-Hopf algebra.
	Moreover, a method to compute partial actions of a given Hopf algebra on its base field is developed and, as an application, we exhibit all partial actions of such type for some families of Hopf algebras.
\end{abstract}



\thanks{The third author was partially supported by CNPq, Brazil.}

\maketitle

\tableofcontents

\section{Introduction}
The first notion of a partial action that appeared in the literature was a partial action of groups on algebras within the theory of operator algebras, more precisely in Exel's paper \cite{exel1994circle} whose aim was to describe the structure of suitable $C^*$-algebras with the unitary circle $\mathbb{S}_1$ acting by automorphisms.
In \cite{Dokuchaev_exel}, Dokuchaev and Exel presented the definition of a group acting partially on an algebra, carrying it to a purely algebraic context. From this previous work, the partial group actions raised the attention of many other researchers.
In particular, a Galois theory for commutative ring extensions was developed using partial actions \cite{paques_ferrero_dokuchaev}. This work motivated Caenepeel and Janssen to extend it to the context of Hopf algebras \cite{caenepeel2008partial}.

The notion of a partial action has several applications and extensions in other branches of mathematics, such as groupoids \cite{paques_bagio}, categories \cite{Alvares_Alves_Batista}, and many others, as can be noticed in \cite{Dokuchaev_survey} and references therein.
In the partial Hopf actions theory, examples constitute a fundamental part of its understanding and development.

An interesting example of partial action is given by scalars, that is, a Hopf algebra $H$ acting partially on a $\k$-algebra $A$ via a partial action of $H$ on $\k$: if $\cdot : H \o \k \longrightarrow \k$ is a partial action of $H$ on $\k$, then $H$ acts partially  on the algebra $A$ via $\rightharpoonup: H \o A \longrightarrow A$, $h  \rightharpoonup  a = (h \cdot 1_\k)a$.
Many authors have obtained a characterization for this particular case of partial action, even in more general settings.
For instance, see \cite[Lemma 4.1]{guris},
\cite[Section 3.2]{Alvares_Alves_Batista}
and \cite[Examples 2.3 - 2.5]{glob_twisted_patial_hopf}.
However, even with this characterization, few concrete examples are known.

Actions of Hopf algebras on algebras are applied to construct crossed products.
Given a global action of $H$ on $A$, the corresponding smash product algebra $A \# H$ is the vector space $A \o H$ with the algebra structure involving such an action.
In particular, since the unique global action of $H$ on $\k$ is given by the counit of $H$, the smash product algebra  $\k \# H$ coincides with the algebra $H$, that is, there are no proper subalgebras of $H$ using global actions in this way.
However, in the partial setting, \emph{i.e.} with a partial (not global) action of $H$ on $\k$, the partial smash product algebra  $\underline{\k \# H}$ is a proper subalgebra of $H$.

In this work, we deal with partial smash product algebras of a Hopf algebra $H$ and its base field $\k$. More precisely, we study a special type of partial matched pair of Hopf algebras using the concept of $\lambda$-Hopf algebras.
Aiming this goal partial actions of $H$ on $\k$ are investigated and a lot of examples are provided.

\medbreak

This paper is organized as follows.
First, we recall in Section \ref{sec:preliminaries} some definitions and results about partial actions of a Hopf algebra $H$ on an algebra $A$.
In particular, the partial smash product algebra $\underline{A \# H}$ is presented.
We focus on the particular case when $A$ is the base field $\k$ of $H$.
Some guiding examples are presented and important notations are fixed.
Then, we dedicate Section \ref{Sec:metodo e exemplos} to study partial actions of a Hopf algebra over its base field.
First, some properties and a method to compute such partial actions are developed, then examples are exhibited.
Finally, in Section \ref{sec:lda_Hopf_Alg} we investigate partial smash product algebras. We obtain a characterization of the conditions under which such an algebra results in a Hopf algebra.
The concept of a $\lambda$-Hopf algebra is introduced, some results concerning it are proved and several examples are calculated using the previously computed partial actions.

\subsection*{Conventions}
In this work, we deal only with algebras over a fixed field $\k$. Unadorned $\otimes$ means $\otimes_{\k}$.
We write $G(H) = \{ g \in H\setminus \{0\} \ | \ \Delta(g)=g \o g\}$ for the group of the \emph{group-like elements} of a Hopf algebra $H$.
Given $g,h \in G(H)$, an element $x \in H$ is called a \emph{$(g,h)$-primitive element} if $\Delta(x) = x \o g + h \o x$, and the linear space of all $(g,h)$-primitive elements of $H$ is denoted by $P_{g,h}(H)$.
If no emphasis on the elements $g,h \in G(H)$ is needed, an element $x \in P_{g,h}(H)$ is called simply of a \emph{skew-primitive element}.
Clearly, $\k \{g-h\} \subseteq P_{g,h}(H)$; these are called the trivial ones.
Finally, given two algebras $A, B,$ we recall that there exist natural inclusions $A, B \hookrightarrow A \o B;$ then, for $a \in A$ and $b \in B$, by $a, b \in A \o B$ we mean $a \o 1_B, 1_A \o b \in A \o B$.

\section{Preliminaries}\label{sec:preliminaries}

We recall here some definitions and preliminary results about partial actions of a Hopf algebra $H$ on an algebra $A$.
Some examples are presented and notations are fixed.
For details, we refer to \cite{Alvares_Alves_Batista, enveloping, caenepeel2008partial}.

\medbreak

	A \emph{left partial action of a Hopf algebra $H$ on an algebra $A$} is a linear map $\cdot: H \otimes A \longrightarrow A$, denoted by $\cdot (h \o a) = h \cdot a$, such that:
	\begin{enumerate}
		\item [(i)] $1_H \cdot a=a$;
		\item [(ii)] $h\cdot ab=(h_1\cdot a)(h_2\cdot b)$;
		\item [(iii)] $h\cdot(k\cdot a)=(h_1\cdot 1_A)(h_2k\cdot a)$,
	\end{enumerate}
	for all $h,k\in H$ and $a,b\in A$.
	In this case, $A$ is called a \emph{partial $H$-module algebra}.
	
	A left partial action is \emph{symmetric} if in addition we have
	\begin{enumerate}
		\item [(iv)] $h \cdot ( k \cdot a)=(h_1k \cdot a)(h_2 \cdot 1_A)$,
	\end{enumerate}
	for all $h,k\in H$ and $a\in A$.

\medbreak

Every global action is a partial action.
The definition of a \emph{right partial action} is given analogously.
Since throughout this work we deal only with left partial actions, from now on by a \emph{partial action} we mean a left partial action.

\medbreak

There exists a characterization for partial actions of $H$ on $\k$, where $\k$ is the base field of the Hopf algebra $H$, given by a linear map $\lambda$ of $H^*=Hom_\k(H,\k)$ satisfying some properties.
This map gives rise to an important source of examples for the partial action theory in several settings: partial actions of group algebras on (co)algebras, partial actions of weak Hopf algebras on (co)algebras, and partial actions of Hopf algebras on categories (see  \cite{glob_twisted_patial_hopf, enveloping, hopf_categories, guris} for details).

The characterization given in the next result is more general. In fact, in \cite{guris} the authors suppose that $H$ is a weak Hopf algebra.
However, we present this result with the assumption that $H$ is a Hopf algebra since we deal only with such algebras in this work.
\begin{prop}\cite[Lemma 4.1]{guris}\label{k_mod_alg_parc}
	Let $H$ be a Hopf algebra.
	Then, $\k$ is a partial $H$-module algebra if and only if there exists a linear map $\lambda \in H^*$ such that $\lambda(1_H)=1_\k$ and
	\begin{align}\label{eqn_lda}
		\lambda(h)\lambda(k)=\lambda(h_1)\lambda(h_2k),
	\end{align}
	for all $h, k \in H$.
	Moreover, the corresponding partial action is symmetric if and only if $\lambda$ satisfies the additional condition
	\begin{align}\label{eqn_lda_sim}
		\lambda(h)\lambda(k)=\lambda(h_1k)\lambda(h_2),
	\end{align}
	for all $h, k \in H$.
\end{prop}

From now on, by a partial action of a Hopf algebra $H$ on its base field $\k$ we mean a linear map $\lambda\in H^*$ such that $\lambda(1_H)=1_\k$ satisfying condition \eqref{eqn_lda}.

Observe that if $\lambda \in H^*$ is a partial action of $H$ on $\k$, then clearly $\lambda$ is an idempotent element in the convolution algebra $H^*$.

\begin{exa}\cite[Example 3.7]{Alvares_Alves_Batista}\label{lda_kg} Let $G$ be a group. Partial actions of $\k G$ on $\k$ are parametrized by subgroups of $G$.
	Indeed, if $N$ is a subgroup of $G$, then $\delta_N$ is a partial action of $\k G$ on $\k$, where $\delta_N: \k G \longrightarrow \k $ is given by $\delta_N (g) = \begin{cases}
	1_\k, & \textrm{if } g \in N \\ 0, & \textrm{if } g \notin N\end{cases},$
	for all $g \in G$.
	Conversely, let $\lambda$ be a partial action of $\k G$ on $\k$. Then, $\lambda(g) \in \{0, 1_\k\}$ and the subset $N = \{g \in G \, | \, \lambda(g) \neq 0 \} = \{g \in G \, | \, \lambda(g) =1_\k \}$ is a subgroup of $G$. Thus, $\lambda=\delta_N$.
\end{exa}

Using the correspondence of the previous example, given $N$ a subgroup of $G$, we denote by $\lambda_N$ the partial action of $\k G$ on $\k$, given as $\lambda_N= \delta_N$.
In particular, $\lambda_G = \varepsilon_{\k G}$ is the global action of $\k G$ on $\k$.

\begin{exa}\cite[Subsection 3.2]{Alvares_Alves_Batista}\label{lda_kg_est}
	Let $G$ be a finite group.
	Then, partial actions of $(\k G)^*$ on $\k$ are in one-to-one correspondence with
	$$\{ N \ | \ N \textrm{ is a subgroup of } G \textrm{ and } char(\k) \nmid |N| \}.$$
	Indeed, let $\{g^* \ | \ g \in G\}$ be the dual basis for the canonical basis $\{g \ | \ g \in G\}$ of $\k G$.
	Then, for each subgroup $N$ of $G$ such that $char(\k) \nmid |N|$, the linear map $\lambda^N : (\k G)^* \longrightarrow \k$ given by
	$\lambda^N (g^*) = \begin{cases}
	1_\k / |N|, & \textrm{ if } g \in N \\ 0, & \textrm{ if } g \notin N\end{cases},$
	for all $g \in G$, is a partial action of $(\k G)^*$ on $\k$.
	Conversely, let $\lambda$ be a partial action of $(\k G)^*$ on $\k$.
	Then, the subset $N = \{g \in G \, | \, \lambda(g^*) \neq 0 \}$ is a subgroup of $G$.
	In this case, it is known that $\lambda(g^*) = 1_\k / |N|$, for all $g \in N$, and so $\lambda = \lambda^N$.
\end{exa}

\begin{exa}\cite[Proposition 4.10]{caenepeel2008partial}\label{lda_Ug} Let $\mathfrak{g}$ be a Lie algebra and $U(\mathfrak{g})$ its universal enveloping algebra.
	Then, every partial $U(\mathfrak{g})$-module algebra $A$ is in fact a global one.
	That is, there is no genuine partial action (\emph{i.e.}, a partial action that is not a global one) of $U(\mathfrak{g})$ on any algebra $A$.
	In particular, if $\lambda$ is a partial action of $U(\mathfrak{g})$ on $\k$, then $\lambda$ is  necessarily the counit of $U(\mathfrak{g})$.
\end{exa}

\begin{exa}\cite[Example 3.8]{Alvares_Alves_Batista}\label{lda_sweedler} Consider $\mathbb{H}_4$ the Sweedler's Hopf algebra. Precisely, $\mathbb{H}_4$ is the algebra over the field $\k$, $char(\k) \neq 2$, generated by the letters $g$ and $x$ satisfying the relations $g^2=1, x^2=0$ and $xg =-gx$.
	The set $\{1, g, x, gx\}$ is a basis for $\mathbb{H}_4$, $g$ is a group-like element and $x$ is a $(1,g)$-primitive element.
	Thus, it is straightforward to check that  for any $\alpha \in \k$ the map $\lambda_\alpha : \mathbb{H}_4 \longrightarrow \k$ given by $\lambda_\alpha (1)=1_\k, \lambda_\alpha(g)=0$ and $\lambda_\alpha(x)=\lambda_\alpha(gx)= \alpha$ is a partial action of $\mathbb{H}_4$ on $\k$.
	Moreover, these are all partial actions of $\mathbb{H}_4$ on $\k$, \emph{i.e.}, if $\lambda$ is a partial action of $\mathbb{H}_4$ on $\k$, then $\lambda=\varepsilon$ (and so the global one) or $\lambda=\lambda_\alpha$ for some $\alpha \in \k$.
\end{exa}

Recall that given any global $H$-module algebra $A$, via an action $\triangleright : H \otimes A \longrightarrow A$, we can endow the tensor product $A \otimes H$ with an algebra structure induced by
$$(a\otimes h)(b\otimes g)=a(h_1 \triangleright b)\otimes h_2g,$$
for $a,b\in A$ and $h,g\in H$. This structure is denoted by $A \# H$ and called \emph{the smash product algebra of $A$ with $H$}.
We typically denote the element $a \otimes h$ by $a \# h \in A \# H$.
In particular, note that $\k \# H \simeq H$.

The same construction can be performed in the partial case. But, in this case, it turns that $A \# H$ is not a unital algebra since $1_A\# 1_H$ is only a left unit.

\begin{definition}\label{algebra_smash} \cite{caenepeel2008partial}
	Let $A$ be a partial $H$-module algebra via $\cdot : H \otimes A \longrightarrow A$. Then, the vector subspace $\underline{ A \# H}=(A\# H)(1_A\# 1_H)$ is a unital algebra with the multiplication induced by
	$$\underline{(x\#h)} \ \underline{(y\#g)}=\underline{x(h_1\cdot y)\# h_2g},$$
	for all $x,y\in A$, $h,g\in H$, where $\underline{(x\#h)} = (x \# h)(1_A \# 1_H) = x(h_1\cdot 1_A)\# h_2$ is a typical element of $\underline{ A \# H}$ and the unit is given by $\underline{1_A\#1_H}$. The unital algebra $\underline{ A \# H}$ is called the \emph{partial smash product algebra of $A$ with $H$}.
\end{definition}

For a partial action of $H$ on $\k$ given by $\lambda: H \longrightarrow \k$, the corresponding partial smash product algebra $\underline{ \k \# H}$ is generated by $\left\{ \underline{ 1_\k \# h} =  1_\k \# \lambda(h_1)h_2 \, | \, h \in H \right\}$.
Also notice $ \underline{ 1_\k \# h} \, \, \underline{ 1_\k \# k} = \underline{ 1_\k \# \lambda(h_1)h_2k}$, for all $h, k \in H$.

For example, let $\lambda_N$ be the partial action of $\k G$ on $\k$ given in Example \ref{lda_kg}.
In this case, it follows that $\underline{ \k \# \k G } \simeq \k N$.

\section{Partial actions of a Hopf algebra on its base field}\label{Sec:metodo e exemplos}

So far we know, only the partial actions on the base field presented in Examples \ref{lda_kg} - \ref{lda_sweedler} are given in the literature.
For this reason, in this section, some properties and a method to compute such partial actions are developed, and then we exhibit these partial actions for some families of Hopf algebras.

\subsection{Properties of a partial action}\label{Subsec:propriedades_lda}

This subsection is dedicated to present some properties of a partial action $\lambda: H \longrightarrow \k$.
Most of the results are about the values that the map $\lambda$ can admit for group-like and skew-primitive elements.

\medbreak

The following proposition generalizes to any Hopf algebra $H$ an important fact observed in Example \ref{lda_kg} just for group algebras.
\begin{prop}\label{N_subgrupo}
	Let $\lambda: H \longrightarrow \k$ be a partial action.
	Then, for any $g \in G(H)$, $\lambda(g) \in \{0, 1_\k\}$ and $N=\{g \in G(H) \, | \, \lambda(g)=1_\k \}$ is a subgroup of $G(H)$.
	In particular, if $g \in N$, then $\lambda(g^i)=1_\k,$ for each $i \in \Z$.
	Moreover, if $g \in G(H)$ has prime order $p$ and $\lambda(g)=0$, then $\lambda(g^i)=0,$ for each $i \in \Z$ such that $gcd(i,p)=1$.
\end{prop}
\begin{proof}
	First, for any $g \in G(H)$, $\lambda(g)\lambda(1_H)=\lambda(g)\lambda(g 1_H)$ by \eqref{eqn_lda}.
	Since $\lambda(1_H)=1_\k$, it follows that $\lambda(g)=\lambda(g)^2$ and so $\lambda(g) \in \{0, 1_\k\}$.
	Now, clearly $1_H \in N$.
	Let $g,h \in N$, then
	we obtain that $\lambda(h)\lambda(h^{-1})=\lambda(h)$ by \eqref{eqn_lda},
	and consequently $h^{-1} \in N$.
	Again, using \eqref{eqn_lda}, we have that $\lambda(g)\lambda(h)=\lambda(g)\lambda(gh)$, and hence $gh \in N$.
	Thus, $N$ is a subgroup of $G(H)$.
	If $g \in N$, then clearly $\lambda(g^i)=1_\k$, since $\langle g \rangle \subseteq N$, where $\langle g \rangle$ stands for the subgroup generated by $g$.
		
	Now suppose that $g^p=1$ and $\lambda(g)=0$, where $p \in \Z$ is a prime number.
	In this case, $g \notin N$ and $\langle g \rangle = \{1, g, g^2, ..., g^{p-1} \}$. Consequently,
	we have that $N \cap \langle g \rangle = \{1\}$ and thus $\lambda(g^i)=0$, whenever $gcd(i,p)=1$.
\end{proof}

The next result shows the behavior of a partial action $\lambda: H \longrightarrow \k$ for certain elements.

\begin{lema}\label{propriedades_lda}
	Let $\lambda: H \longrightarrow \k$ be a partial action.
	For $g,t \in G(H)$ and $x \in P_{g,t}(H)$, it holds that:
	\begin{itemize}
		\item[a)]  If $\lambda(g)=1$, then $\lambda(gu)=\lambda(u)$, for all $u \in H$;
		\item[b)]  If $\lambda(g)=\lambda(t)$, then $\lambda(x)=0$;
		\item[c)]  If $\lambda(x)=0$ and $\lambda(t)=1$, then $\lambda(xu)=0$, for all $u \in H$;
		\item[d)] If $\lambda(g)=1$ and $\lambda(t)=0$, then $\lambda(xt^{-1})=-\lambda(x)$;
		\item[e)] If $\lambda(g)=0$ and $\lambda(t)=1$, then $\lambda(xg^{-1})=-\lambda(x)$.
	\end{itemize}
	Moreover, if $gt=tg$, then, for all $i \in \Z$, we also have that:
	\begin{itemize}
		\item[f)]  If $\lambda(g)=1$ and $\lambda(t)=0$, then $\lambda(t^{i}x)=0 $ or $\lambda(t^i)+\lambda(t^{i+1})=1$;
		\item[g)]  If $\lambda(g)=0$ and $\lambda(t)=1$, then  $\lambda(g^{i}x)=0 $ or $\lambda(g^i)+\lambda(g^{i+1}) =1$;
		\item[h)] If $\lambda(g)=\lambda(t)=0$, then $\lambda(g^{-1}t^{-1}x)=0$.
	\end{itemize}
\end{lema}
\begin{proof}
	It follows directly from the application of equality \eqref{eqn_lda} for good choices of elements $h, k \in H$ and Proposition \ref{N_subgrupo}.
\end{proof}

\subsection{About the computations of a partial action of $H$ on $\k$}\label{Subset:metodo}

Let $H$ be a Hopf algebra.
In this subsection, a method to compute partial actions of $H$ on its base field is given.

\medbreak

Recall by Proposition \ref{k_mod_alg_parc} that $\lambda \in H^*$ is a partial action of $H$ on $\k$ if and only if the map $\lambda$ satisfies $\lambda(1_H)=1_\k$ and $\lambda(u)\lambda(v)= \lambda(u_1)\lambda(u_2v)$, for all $u, v \in H$.
\begin{definition}
Let $\mathcal{B}$ be a basis of $H$ and consider the linear map $\Lambda: H \longrightarrow \k [ X_b \ | \ b \in \B ]$, $\Lambda(b) = X_b$.
	We define by \emph{partial system associated with $\B$} the following system of equations:
\begin{align}\label{sistema_parcial_associado}
	\left\{ \begin{array}{ll}\Lambda(1_H) = 1_\k \\
		\Lambda(a)\Lambda(b)= \Lambda(a_1)\Lambda(a_2 b)  \end{array}\right. \qquad a, b \in \B.
\end{align}
\end{definition}
Hence, the map $\lambda : H \longrightarrow \k$ is a partial action of $H$ on $\k$ if and only if $(\lambda(b))_{b \in \B}$ is a solution of \eqref{sistema_parcial_associado}.
In particular, the partial system associated with $\B$ always has at least one solution:
the global action of $H$ on $\k$, namely $(\varepsilon(b))_{b \in \B}$.

\medbreak

Now, extend $G(H)$ to a basis $\B = G(H) \sqcup \B'$ of $H$, where $\sqcup$ denotes the disjoint union.
Then, \eqref{sistema_parcial_associado} is rewritten as
\begin{align*}
\left\{\begin{array}{ll} \Lambda(1_H)= 1_\k \\
\Lambda(g)\Lambda(v)= \Lambda(g)\Lambda(g v) \\
\Lambda(u)\Lambda(v)= \Lambda({u}_1)\Lambda({u}_2 v) \end{array}\right. \qquad g \in G(H), u \in \B', v \in \B.
\end{align*}

Let $\lambda:H \longrightarrow \k$ be a partial action.
Proposition \ref{N_subgrupo} ensures that $\lambda(h) \in \{0, 1_\k\}$, for all $h \in G(H)$, and there exists a subgroup $N$ of $G(H)$ such that $\lambda |_{\k G(H)} = \lambda_N$ (see Example \ref{lda_kg}).
Then, by Lemma \ref{propriedades_lda} (a), we have that $\lambda(v) = \lambda(g v)$ for all $g \in N, v \in \B$.
Therefore, $(\lambda(b))_{b \in \B}$ is a solution of the following system:
\begin{align}\label{sistema_condicao_inicial}
\left\{\begin{array}{ll} \Lambda(g) = 1_\k \\
\Lambda(h) = 0 \\
\Lambda(u) = \Lambda(g u) \\
\Lambda(u)\Lambda(v)= \Lambda({u}_1)\Lambda({u}_2 v) \end{array}\right. \qquad g \in N, h \in G(H) \setminus N, u \in \B', v \in \B.
\end{align}

On the other hand, let $N$ be a subgroup of $G(H)$ and consider a system of equations as \eqref{sistema_condicao_inicial}.
If there exists a solution $(\alpha_b)_{b \in \B}$ to the latter system, then it is also a solution of the partial system associated with $\B$.

\medbreak

Fix a subgroup $N$ of $G(H)$.
For $x, y \in \B'$, define the relation
$$ x \sim_N y  \quad \textrm{ if and only if} \quad  \exists  \, \,  g \in N  \textrm{ such that } y = gx.$$
Notice that it is an equivalence relation on the set $\B'$.
Given $ x \in \B'$, $[x] = \{y \in \B' \ | \ y \sim_N x \}$ denotes the \emph{equivalence class of $x$}.
In particular, for any $x \in \B'$, $[x] = \{x\} \sqcup [x]^{\perp}$, where $[x]^{\perp} = \{y \in \B' \ | \ y \sim_N x, y \neq x \}$.

Write $\widetilde{N}$ for a \emph{transversal set of the relation $\sim_N$ on $\B'$}, that is, $\widetilde{N}$ is a subset of $\B'$ consisting of exactly one representative from each equivalence class.
Then,  we have a partition of $\B'$ given by $\B' = \widetilde{N} \sqcup N^{\perp}$, where $N^{\perp}=\B' \setminus \widetilde{N}$.
Observe that $N^{\perp} = \cup_{x \in \widetilde{N}} [x]^{\perp}$.

Consider now the following system of equations:
\begin{align}\label{sistema_condicao_inicial_red}
\left\{\begin{array}{ll} \Lambda(g) = 1_\k \\
\Lambda(h) = 0 \\
\Lambda(u) = \Lambda(g u)\end{array}\right.
\qquad g \in N, h \in G(H) \setminus N, u \in \B'.
\end{align}
Then, $(\alpha_b)_{b \in \B}$ is a solution of \eqref{sistema_condicao_inicial_red} if and only if $\alpha_g= 1_\k$, $\alpha_h= 0$, $\alpha_y=\alpha_x$, and $\alpha_x \in \k$ is a free parameter, for all $g \in N$, $h \in G(H)\setminus N$, $x \in \widetilde{N}$ and $y \in [x]$.
In this case, we say that $(\alpha_b)_{b \in \B}$ is an \emph{initial $N$-condition for the partial system associated with $\B$}.

Finally, we define the following sets: $\B_{t,s} = \left\{ x \in P_{t,s}(H) \ | \ t \in N,  s \in G(H) \setminus N \right\}$ and
$\widetilde{\B} = \widetilde{N} \setminus \left(\cup_{t,s \in G(H)} \B_{t,s}\right).$
With these notations, the system
\begin{align}\label{N_reduzido}
\left\{\begin{array}{ll}
\Lambda(u)\Lambda(v)= \Lambda({u}_1)\Lambda({u}_2 v)
\end{array}\right. \qquad u \in \widetilde{\B}, v \in \B,
\end{align}
is called of the \emph{$N$-reduced partial system associated with $\B$}.

\begin{thm}\label{acao_e_sistema}
	Let $\B = G(H) \sqcup \B'$ be a basis of $H$ and $\lambda: H \longrightarrow \k$ a linear map.
	Then, $\lambda$ is a partial action of $H$ on $\k$ if and only if $N=\{g \in G(H) \ | \ \lambda(g)=1_\k\}$ is a subgroup of $G(H)$ and $(\lambda(b))_{b \in \B} $ is both an initial $N$-condition for the partial system associated with $\B$ and also a solution of the $N$-reduced partial system associated with $\B$.
\end{thm}
\begin{proof}
	If $\lambda$ is a partial action of $H$ on $\k$, then we have that $N$ is a subgroup of $G(H)$ by Proposition \ref{N_subgrupo}, and $(\lambda(b))_{b \in \B}$ is a solution of the system \eqref{sistema_condicao_inicial}. Then, clearly $(\lambda(b))_{b \in \B}$ is both: an initial $N$-condition for the partial system associated with $\B$ and a solution of the $N$-reduced partial system associated with $\B$.
	
	For the converse, it is sufficient to check that $(\lambda(b))_{b \in \B}$ satisfies each equation $ \Lambda(u)\Lambda(v)= \Lambda({u}_1)\Lambda({u}_2 v)$, for $u \in \B' \setminus \widetilde{\B}, v \in \B.$ Then, in this case, $(\lambda(b))_{b \in \B}$ is also a solution of the system given in \eqref{sistema_condicao_inicial} and therefore $\lambda$ is a partial action.
	
	First, suppose $u \in N^{\perp}.$
	Note that $u \in [x]^{\perp}$ for some $x \in \widetilde{N}$, and so there exists $g \in N$ such that $u = gx$.
	Then,
	$ \Lambda(u)\Lambda(v)= \Lambda({u}_1)\Lambda({u}_2 v)$ means $\Lambda(gx)\Lambda(v) = \Lambda(gx_1)\Lambda(gx_2v)$.
	Since $(\lambda(b))_{b \in \B}$ is an initial $N$-condition for the partial system associated with $\B$, we have $\lambda(gw)=\lambda(w)$ for any $w \in H$, and
	the equality $\lambda(x)\lambda(v) = \lambda(x_1)\lambda(x_2v)$  holds because $(\lambda(b))_{b \in \B}$ is solution of \eqref{N_reduzido}.
	Hence,
	$$\lambda(gx)\lambda(v) = \lambda(x)\lambda(v) = \lambda(x_1)\lambda(x_2v) = \lambda(gx_1)\lambda(gx_2v).$$

	Finally, if $ u \in \B_{t,s} \cap \widetilde{N}$, then the equation $\Lambda(u)\Lambda(v) = \Lambda(u_1)\Lambda(u_2v)$ is the same as $\Lambda(u)\Lambda(v) = \Lambda(u)\Lambda(tv) + \Lambda(s)\Lambda(uv)$, for any $v \in \B$.
	Since $(\lambda(b))_{b \in \B}$ is a solution of \eqref{sistema_condicao_inicial_red}, then $\lambda(v) = \lambda(tv)$ and $\lambda(s)=0$.
	Thus, the equality $\lambda(u)\lambda(v) = \lambda(u)\lambda(tv) + \lambda(s)\lambda(uv)$ holds.
\end{proof}

\medbreak

Similarly to Lemma \ref{propriedades_lda}, we have the following result.
\begin{lem}\label{prop_analogas_lda}
	Let $\B = G(H) \sqcup \B'$ be a basis of $H$, $g, t \in G(H)$, $x \in P_{g,t}(H) \cap \B'$ and $(\alpha_b)_{b \in \B}$ an initial $N$-condition for the partial system associated with $\B$.
	If $(\alpha_b)_{b \in \B}$ is also a solution of the $N$-reduced partial system associated with $\B$, then:
	\begin{itemize}
		\item[a)] If $\alpha_g = \alpha_t$, then $\alpha_x=0$;
		\item[b)] If $\alpha_t = 1_\k$ and $\alpha_x=0$, then $\alpha_{xu}=0$ for all $u \in H$ such that $xu \in \B'$;
		\item[c)] If $\alpha_g = 1_\k$, $\alpha_t = 0$ and $xt^{-1} \in \B'$, then $\alpha_{xt^{-1}}=-\alpha_{x}$;
		\item[d)] If $\alpha_g = 0$, $\alpha_t = 1_\k$ and $xg^{-1} \in \B'$, then $\alpha_{xt^{-1}}=-\alpha_{x}$.
	\end{itemize}
\end{lem}
\begin{proof}
		Straightforward from Theorem \ref{acao_e_sistema} and Lemma \ref{propriedades_lda}.
\end{proof}

\medbreak
	
	Given a partial action $\lambda: H \longrightarrow \k$, we say that $\lambda$ has \emph{initial condition $N$} if $N=\{ g \in G(H) \ | \ \lambda(g)=1_\k \}$.

\medbreak

Now, based on Theorem \ref{acao_e_sistema}, we present our method to calculate partial actions of $H$ on $\k$.

\smallbreak

\textbf{\underline{The Method:}}
Let $H$ be a given Hopf algebra.
To obtain a partial action of $H$ on $\k$, we perform the following steps:
\begin{itemize}
	\item[\textbf{Step 1.}] Extend $G(H)$ to a basis $\B$ of $H$;
	\item[\textbf{Step 2.}] Consider $N$ a subgroup of $G(H)$;
	\item[\textbf{Step 3.}] Regard the solutions $(\alpha_b)_{b \in \B}$ of the system \eqref{sistema_condicao_inicial_red};
	\item[\textbf{Step 4.}] Investigate the $N$-reduced partial system associated with $\B$;
	\item[\textbf{Step 5.}] \textbf{Conclusion:} If $(\alpha_b)_{b \in \B}$ is a solution of the $N$-reduced partial system associated with $\B$, then the linear map $\lambda: H \longrightarrow \k$ given by $\lambda(b)=\alpha_b$, for all $b \in \B$, is a partial action of $H$ on $\k$;
	Otherwise, there is not a partial action of $H$ on $\k$ with initial condition $N$.
\end{itemize}

\medbreak

\textbf{Remarks. 1)} For Step 4, Lemma \ref{prop_analogas_lda} is useful to determine if $(\alpha_b)_{b \in \B}$ can be a solution of \eqref{N_reduzido} or not, since it imposes necessary conditions on the values $\alpha_u \in \k$ for some elements $u \in \widetilde{N}$.

\textbf{2)} Since every partial action of $H$ on $\k$ has an initial condition, if we apply the method to all subgroups of $G(H)$, then we obtain all partial actions of $H$ on $\k$.
Moreover, the method gives an explicit characterization of each partial action $\lambda: H \longrightarrow \k$.

\medbreak

\begin{exa}\label{metodo_sweedler}
	To illustrate the method, we reobtain the partial actions of the Sweedler's Hopf algebra $\mathbb{H}_4$ on its base field, previously exhibited in Example \ref{lda_sweedler}.
	Recall that $G(\mathbb{H}_4) = \{1,g \}$ and $\mathcal{B} = \{1,g,x,gx \}$ is a basis of $\mathbb{H}_4$.
	Thus, we have two subgroups of $G(H)$, namely $N=\{1, g\}$ and $N=\{1\}$.
	\begin{enumerate}
		\item $N=\{1, g\}$.
		In Step 3 we assume $\alpha_1=\alpha_g=1_\k$ and $\alpha_{gx}=\alpha_x \in \k$.
		Now, Step 4 says to study \eqref{N_reduzido}.
		But, using Lemma \ref{prop_analogas_lda} (a), if $(\alpha_b)_{b \in \B}$ is a solution of \eqref{N_reduzido}, then necessarily $\alpha_x=0$.
		In this case, the only possibility for the associated partial action $\lambda : \mathbb{H}_4 \longrightarrow \k$, $\lambda(b) = \aa_b$, is the counit $\varepsilon$.
		Thus, as conclusion, $\lambda$ is a partial action of $\mathbb{H}_4$ on $\k$ with initial condition $\{1, g\}$ if and only if $\lambda = \varepsilon$;
		
		\item  $N=\{1\}$.
		In this case, by Step 3 we suppose $\alpha_1=1_\k$, $\alpha_g=0$ and $\alpha_x, \alpha_{gx} \in \k$.
		Now, we move to Step 4.
		Note that $\widetilde{N}=\{x, gx\}$.
		As $x \in \B_{1,g}$,
		it follows that $\widetilde{\B} = \{gx\}$ and the system \eqref{N_reduzido} has only $4$ equations:
		\begin{align*}
		\left\{ \Lambda(gx)\Lambda(v)=\Lambda(gx)\Lambda(gv) + \Lambda(1)\Lambda(gxv) \right. \qquad v \in \mathcal{B}.
		\end{align*}
		To verify if $(\alpha_b)_{b \in \B}$ is a solution of the system above, we analyze
		$$\left\{\begin{array}{ll} \alpha_{gx}=\alpha_{gx} \\
		0=\alpha_{gx} - \alpha_{x} \\
		\alpha_{gx}\alpha_{x} = \alpha_{gx}\alpha_{gx} \\
		\alpha_{gx}\alpha_{gx}=\alpha_{gx}\alpha_{x}. \end{array}\right. $$
		Thus,  $(\alpha_b)_{b \in \B}$ is a solution of the $N$-reduced partial system associated with $\B$ if and only if $\alpha_{x}=\alpha_{gx}$.
		
		Consequently, for any $\alpha \in \k$, the linear map $\lambda_\alpha: \mathbb{H}_4 \longrightarrow \k$ given by $\lambda(1)=1_\k$, $\lambda(g)=0$ and $\lambda(x)=\lambda(gx)= \alpha$, is a partial action of $\mathbb{H}_4$ on $\k$.
	\end{enumerate}
Therefore, the partial actions of $\mathbb{H}_4$ on $\k$ are exactly $\varepsilon$ and $\lambda_\alpha$, $\alpha \in \k$.
\end{exa}

\medbreak

	We point out the benefit of our method: since a partial action of $H$ on $\k$ corresponds to a linear map $\lambda \in H^*$ satisfying $\lambda(1_H)=1_\k$ and \eqref{eqn_lda}, to calculate explicitly such a map is a task attached to solve the system \eqref{sistema_parcial_associado}, for some basis $\B$ of $H$.
	However, with our method, we deal only with the potentially smaller system \eqref{N_reduzido}.
	In Sweedler's Hopf algebra, for instance, the system \eqref{sistema_parcial_associado} has $4^2=16$ equations, while we deal only with the system of 4 equations \eqref{N_reduzido}, in the case $N=\{1\}$.

\medbreak

Since our developments deal with group-like and skew-primitive elements, it is expected that the method will be effective mainly for a Hopf algebra generated by elements of these types, as illustrated through the example above.

Using the method described, we were able to determine all partial actions on the base field of some particular families of Hopf algebras,
as can be seen in the following subsection.

\subsection{Partial actions of the pointed non-semisimple Hopf algebras $H$ with $dim_\k(H)=8,16$}\label{Subsec: 8}

In this subsection, for each pointed non-semisimple Hopf algebra $H$ with $dim_\k(H)=8,16$, we exhibit all partial actions of $H$ on its base field $\k$.
Throughout this subsection, $\k$ is assumed to be
an algebraically closed field of characteristic zero.

\medbreak

The 8-dimensional Hopf algebras were classified independently by \cite{williams} and \cite{stefan}.
There exist five (up to isomorphism) 8-dimensional pointed non-semisimple Hopf algebras.
We use the notations and  presentations for such Hopf algebras as given in \cite{classifying}:

\medbreak

$\mathcal{A}_{2} =  \langle g, \, x, \, y  \,  |  \, g^2 =1, \,\, x^2=y^2=0, \,\, xg = -gx, \,\, yg=-gy, \,\, yx=-xy \rangle ,$ where $g \in G(\mathcal{A}_{2})$ and $x, y \in P_{1,g}(\mathcal{A}_{2})$;

\smallbreak

$\mathcal{A}_{4}' = \langle g, \, x \,  | \, g^4 =1, \,\, x^2=0, \,\, xg = -gx \rangle ,$ where $g \in G(\mathcal{A}_{4}')$ and $x \in P_{1,g}(\mathcal{A}_{4}')$;

\smallbreak

$\mathcal{A}_{4}'' = \langle g, \, x \, | \, g^4 =1, \, x^2=g^2-1, \, xg=-gx\rangle ,$ where $g \in G(\mathcal{A}_{4}'')$ and $x \in P_{1,g}(\mathcal{A}_{4}'')$;

\smallbreak

$\mathcal{A}_{4,q}''' = \langle g, \, x \, | \,  g^4 =1, \,\, x^2=0, \,\, gx = qxg \rangle,$ where $q$ is a primitive root of unity of order $4$, $g \in G(\mathcal{A}_{4,q}''')$ and $x \in P_{1,g^2}(\mathcal{A}_{4,q}''')$;

\smallbreak

$\mathcal{A}_{2,2} = \langle g, \, h, \, x \,  | \, g^2 = h^2 = 1, \, x^2=0, \, xg = -gx, \, xh = -hx, \, hg=gh \rangle,$ where $g,h \in G(\mathcal{A}_{2,2})$ and $x \in P_{1,g}(\mathcal{A}_{2,2})$.

\medbreak

We observe that the subscript index in each Hopf algebra above described represents its group-like subalgebra.

\begin{prop}\label{teo_lda_dim8}
	For each pointed non-semisimple Hopf algebra of dimension 8, all partial actions on its base field are computed. They are presented in Tables \ref{tab:A_2} - \ref{tab:A_22}, where $\alpha, \beta, \gamma \in \k$ and $\gamma^2=-1$.
	\begin{center}
		\begin{table}[!ht]
		\caption{Partial actions of $\mathcal{A}_2$ on $\k$}\label{tab:A_2}
			\begin{tabular}{| c | c | c | c | c | c | c | c | c | c |}\hline
				& $1$ & $g$ & $x$ & $y$ & $gx$ & $gy$ & $xy$ & $gxy$	\\ \hline
				$\lambda_{G(\mathcal{A}_2)}=\varepsilon_{\mathcal{A}_2}$ & $1$ & $1$ & $0$ & $0$ & $0$ & $0$ & $0$ & $0$ \\ \hline
				$\lambda_{\{1\}}$ &  $1$ & $0$ & $\alpha$ & $\beta$ & $\alpha$ & $\beta$ & $0$ & $0$ \\ \hline
			\end{tabular}		
\end{table}
\begin{table}[!ht]
				\caption{Partial actions of $\mathcal{A}_{4}'$ on $\k$}\label{tab:A_4l}
			\begin{tabular}{| c | c | c | c | c | c | c | c | c | c |}\hline
				& $1$ & $g$ & $g^2$ & $g^3$ & $x$ & $gx$ & $g^2x$ & $g^3x$	\\ \hline
				$\lambda_{G(\mathcal{A}_{4}')}=\varepsilon_{\mathcal{A}_{4}'}$ & $1$ & $1$ & $1$ & $1$ & $0$ & $0$ & $0$ & $0$ \\ \hline
				$\lambda_{\{1, g^2\}}$ &  $1$ & $0$ & $1$ & $0$ & $\alpha$ & $\alpha$ & $\alpha$ & $\alpha$  \\ \hline
				$\lambda_{\{1\}}$ &  $1$ & $0$ & $0$ & $0$ & $0$ & $0$ & $0$ & $0$ \\ \hline
			\end{tabular}
\end{table}
\begin{table}[!ht]
				\caption{Partial actions of $\mathcal{A}_{4}''$ on $\k$ $(\gamma^2=-1)$}\label{tab:A_4ll}
		\begin{tabular}{| c | c | c | c | c | c | c | c | c | c |}\hline
				& $1$ & $g$ & $g^2$ & $g^3$ & $x$ & $gx$ & $g^2x$ & $g^3x$	\\ \hline
				$\lambda_{G(\mathcal{A}_{4}'')}=\varepsilon_{\mathcal{A}_{4}''}$ & $1$ & $1$ & $1$ & $1$ & $0$ & $0$ & $0$ & $0$ \\ \hline
				$\lambda_{\{1,g^2\}}$ &  $1$ & $0$ & $1$ & $0$ & $\alpha$ & $\alpha$ & $\alpha$ & $\alpha$ \\ \hline
				$\lambda_{\{1\}}$ &  $1$ & $0$ & $0$ & $0$ & $\gamma$ & $0$ & $0$ & $\gamma$ \\ \hline
			\end{tabular}\\
		\end{table}
	\begin{table}[!ht]
					\caption{Partial actions of $\mathcal{A}_{4,q}'''$ on $\k$}\label{tab:A_4lll}
			\begin{tabular}{| c | c | c | c | c | c | c | c | c | c |}\hline
				& $1$ & $g$ & $g^2$ & $g^3$ & $x$ & $gx$ & $g^2x$ & $g^3x$	\\ \hline
				$\lambda_{G(\mathcal{A}_{4,q}''')} =\varepsilon_{\mathcal{A}_{4,q}'''}$ & $1$ & $1$ & $1$ & $1$ & $0$ & $0$ & $0$ & $0$ \\ \hline
				$\lambda_{\{1\}}$ &  $1$ & $0$ & $0$ & $0$ & $\alpha$ & $0$ & $\alpha$ & $0$  \\ \hline
				$\lambda_{\{1, g^2\}}$ &  $1$ & $0$ & $1$ & $0$ & $0$ & $0$ & $0$ & $0$ \\ \hline
			\end{tabular}
		\end{table}
	\begin{table}[!ht]
				\caption{Partial actions of $\mathcal{A}_{2,2}$ on $\k$}\label{tab:A_22}
			\begin{tabular}{| c | c | c | c | c | c | c | c | c | c |}\hline
				& $1$ & $g$ & $h$ & $gh$ & $x$ & $gx$ & $hx$ & $ghx$	\\ \hline
				$\lambda_{G(\mathcal{A}_{2,2})}=\varepsilon_{\mathcal{A}_{2,2}}$ & $1$ & $1$ & $1$ & $1$ & $0$ & $0$ & $0$ & $0$ \\ \hline
				$\lambda_{\{1\}}$ &  $1$ & $0$ & $0$  & $0$ & $\alpha$ & $\alpha$ & $0$ & $0$  \\ \hline
				$\lambda_{\{1, gh\}}$ &  $1$ & $0$ & $0$ & $1$ & $\alpha$ &  $\alpha$ &  $\alpha$ &  $\alpha$ \\ \hline
				$\lambda_{\{1, h\}}$ & $ 1$ & $0$ & $1$ & $0$ & $0$ & $0$ & $0$ & $0$  \\ \hline
				$\lambda_{\{1, g\}}$ &  $1$ & $1$ & $0$ & $0$ & $0$ & $0$ & $0$ & $0$ \\ \hline
			\end{tabular}
		\end{table}
	\end{center}
\end{prop}

\begin{proof}
	To obtain the partial actions presented in Tables \ref{tab:A_2} - \ref{tab:A_22}, one does routine computations applying the method presented.
	We will do just some computations to obtain all partial actions of $\mathcal{A}_{4}''$ on $\k$, presented in Table \ref{tab:A_4ll}.
	
	Step 1: $G(\mathcal{A}_{4}'') = \{1,g,g^2,g^3 \}$ and  $\mathcal{B} = \{1,g,g^2, g^3,x,gx, g^2x,g^3x \}$.
	
	For Step 2, we note that $G(\mathcal{A}_{4}'')$ has 3 subgroups: $\{1, g, g^2, g^3\}$, $\{1, g^2\}$ and $\{1\}$.
	\begin{enumerate}
		\item $N=\{1, g, g^2, g^3\}$.
		Here, if we proceed as in Example \ref{metodo_sweedler} (1), we obtain that $\lambda$ is a partial action of $\mathcal{A}_{4}''$ on $\k$ with initial condition $\{1, g, g^2, g^3\}$ if and only if $\lambda=\varepsilon_{\mathcal{A}_{4}''}$;
		
		\item $N=\{1, g^2\}$.
		By Step 3, we consider $\aa_1=\aa_{g^2}=1$, $\aa_g= \aa_{g^3}=0$, $\aa_{g^2x}=\aa_x$ and $\aa_{g^3x}=\aa_{gx}$, where $\aa_x, \aa_{gx} \in \k$.
		For Step 4, since $x \in \B_{1,g}$ we have $\widetilde{\B}=\{gx\}$. Then, the system \eqref{N_reduzido} has 8 equations:
			\begin{align*}
		\left\{ \Lambda(gx)\Lambda(v)=\Lambda(gx)\Lambda(gv) + \Lambda(g^2)\Lambda(gxv) \right. \qquad v \in \mathcal{B}.
		\end{align*}
	Assume that $(\alpha_b)_{b \in \B}$ is a solution of the system above.
		In this case, if we replace $(\alpha_b)_{b \in \B}$ in the system above, we get
		$$\left\{\begin{array}{ll}
		\alpha_{gx} = \alpha_{gx} \\
		0 = \alpha_{gx} - \alpha_{x} \\
		\alpha_{gx}=\alpha_{gx} \\
		0 = \alpha_{gx} - \alpha_{x} \\
		\alpha_{gx}\alpha_{x}=\alpha_{gx}\alpha_{gx} \\
		\alpha_{gx}\alpha_{gx}=\alpha_{gx}\alpha_{x} \\
		\alpha_{gx}\alpha_{x}=\alpha_{gx}\alpha_{gx} \\
		\alpha_{gx}\alpha_{gx}=\alpha_{gx}\alpha_{x}.
		\end{array}\right. $$
		Clearly, $(\alpha_b)_{b \in \B}$ is a solution of the system above if and only if $\aa_{x}=\aa_{gx}$.
		
		Consequently, for any $\alpha \in \k$, the linear map $\lambda_N: \mathcal{A}_{4}'' \longrightarrow \k$ given by $\lambda(1)=\lambda(g^2)=1$, $\lambda(g)=\lambda(g^3)=0$ and $\lambda(x) = \lambda(gx) = \lambda(g^2x) = \lambda(g^3x)= \alpha$, is a partial action of $\mathcal{A}_{4}''$ on $\k$.
		
		\item $N=\{1\}$. By Step 3, we consider  $\aa_1 = 1$, $\aa_g= \aa_{g^2}= \aa_{g^3}=0$ and $\aa_x, \aa_ {gx}, \aa_{g^2x} \aa_{g^3x} \in \k$.
		
		Now we move to Step 4.
		As $x \in \B_{1,g}$, $\widetilde{\B}=\{gx, g^2x, g^3x \}$.
		Then, the system \eqref{N_reduzido} has 3 blocks of 8 equations each:	
		\begin{align*}
		\left\{ \begin{array}{ll} \Lambda(gx)\Lambda(v)=\Lambda(gx)\Lambda(gv) + \Lambda(g^2)\Lambda(gxv) \\
		\Lambda(g^2x)\Lambda(v)=\Lambda(g^2x)\Lambda(g^2v) + \Lambda(g^3)\Lambda(g^2xv) \\
		\Lambda(g^3x)\Lambda(v)=\Lambda(g^3x)\Lambda(g^3v) + \Lambda(1)\Lambda(g^3xv) \end{array} \right. \qquad v \in \mathcal{B}.
		\end{align*}
		Suppose that $(\alpha_b)_{b \in \B}$ is a solution of the system above.
		Then, by Lemma \ref{prop_analogas_lda} (a), we obtain $\aa_{gx}=\aa_{g^2x}=0$.
		Thus, we can discard the first and the second block of equations in the system above, and deal only with the third. That is,
		$$\left\{\begin{array}{ll}
	\aa_{g^3x} = \aa_{g^3x} \\
	0 = \aa_{g^3x} - \aa_{x} \\
	0 = \aa_{gx} \\
	0 = - \aa_{g^2x} \\
	\aa_{g^3x}\aa_{x}=\aa_{g^3x} \aa_{g^3x} \\
	0 =\aa_{g^3x} \aa_{x} + 1 \\
	0 =\aa_{g^3x} \aa_{gx} \\
	\aa_{g^3x}\aa_{g^3x}= - 1.
		\end{array}\right. $$
		Hence, $(\alpha_b)_{b \in \B}$ is a solution of the system above if and only if $\aa_{x}=\aa_{g^3x}$ and $\aa_{x}^2=-1$.
\end{enumerate}

	Therefore, all partial actions of $\mathcal{A}_4''$ on $\k$ are given in Table \ref{tab:A_4ll}.
\end{proof}

\medbreak

The 16-dimensional pointed non-semisimple Hopf algebras were classified by \cite{classif_pointed_dim16}, and there exist 29 isomorphism classes.

We will adopt the notation and presentation for these Hopf algebras given by generators and relations in \cite{almost_inv}, and provide in the following proposition, all the partial actions of these Hopf algebras on their base fields. 
Such partial actions are obtained by routine computations using the method described in Subsection \ref{Subset:metodo}. In particular, note that given a subgroup $N$ of $G(H)$, it is not true that there always exists a partial action $\lambda : H \longrightarrow \k$ with initial condition $N$. See, for instance, Table \ref{tab:H_6}.
In order to reduce the size of the tables, we will omit $\lambda_N(h)$ for all $h \in G(H)$, since they are determined uniquely by $N$.

\begin{prop}\label{teo_lda_dim16}
	For each pointed non-semisimple Hopf algebra of dimension 16, all partial actions on its base field are computed. They are presented in Tables \ref{tab:H_1} - \ref{tab:H_29}, where $\alpha, \beta, \theta, \gamma, \omega, \delta, \sigma, \Omega, \zeta  \in \k$ are such that  $\gamma^2 =  \omega^2 =-1,$    $\delta \sigma =0$ and $\Omega \zeta = -\frac{1}{2}$.

	{ \scriptsize
	\begin{center}
		\vspace{-0.5cm}
		\begin{table}[!ht]
			\caption{Partial actions of $\mathbf{H_{1}}$ on $\k$}\label{tab:H_1}\vspace{-0.45cm}
			\begin{tabular}{ | c | c | c | c | c | c | c | c | c | c | c | c | c | c | c | c | c | }\hline
				& $x$ & $gx$ & $y$ & $gy$ & $z$ & $gz$ & $xy$ & $gxy$ & $xz$ & $gxz$ & $yz$ & $gyz$ & $xyz$ & $gxyz$ \\ \hline
				$\lambda_G\!=\!\varepsilon$ & $0$ & $0$ & $0$ & $0$ & $0$ & $0$ & $0$ & $0$ & $0$ & $0$ & $0$ & $0$ & $0$ & $0$ \\ \hline
				$\lambda_{\{1\}}$ &  $\alpha$ & $\alpha$ & $\beta$ & $\beta$ & $\theta$ &  $\theta$ & $0$ & $0$ & $0$ & $0$ & $0$ & $0$ & $0$ & $0$ \\ \hline
			\end{tabular}\vspace{0.3cm}
%			\end{table}	
%			\begin{table}[!ht]
			\caption{Partial actions of $\mathbf{H_{2}}$ on $\k$}\label{tab:H_2}\vspace{-0.45cm}
			\begin{tabular}{ | c | c | c | c | c | c | c | c | c | c | c | c | c | c | c | c | c | }\hline
				& $x$ & $gx$ & $g^2x$ & $g^3x$ & $y$ & $gy$ & $g^2y$ & $g^3y$ & $xy$ & $gxy$ & $g^2xy$ & $g^3xy$ \\ \hline
				$\lambda_G =\varepsilon$  & $0$ & $0$ & $0$ & $0$ & $0$ & $0$ & $0$ & $0$ & $0$ & $0$ & $0$ & $0$ \\ \hline
				$\lambda_{\{1\}}$ & $\alpha$ & $0$ & $\alpha$ &  $0$ & $\beta$ & $0$ & $\beta$ & $0$ & $0$ & $0$ & $0$ & $0$ \\ \hline
				$\lambda_{\{1, g^2\}}$ & $0$ & $0$ & $0$ &  $0$ & $0$ & $0$ & $0$ & $0$ & $0$ & $0$ & $0$ & $0$ \\ \hline
			\end{tabular}\vspace{0.3cm}
		%	\end{table}	
	%	\begin{table}[!ht]	
			\caption{Partial actions of $\mathbf{H_{3}}$ on $\k$}\label{tab:H_3}\vspace{-0.45cm}
			\begin{tabular}{ | c | c | c | c | c | c | c | c | c | c | c | c | c | c | c | c | c | }\hline
				& $x$ & $gx$ & $g^2x$ & $g^3x$ & $y$ & $gy$ & $g^2y$ & $g^3y$ & $xy$ & $gxy$ & $g^2xy$ & $g^3xy$ \\ \hline
				$\lambda_G=\varepsilon$  & $0$ & $0$ & $0$ & $0$ & $0$ & $0$ & $0$ & $0$ & $0$ & $0$ & $0$ & $0$ \\ \hline
				$\lambda_{\{1\}}$ & $\alpha$ & $0$ & $\alpha$ &  $0$ & $\beta$ & $0$ & $\beta$ & $0$ & $0$ & $0$ & $0$ & $0$ \\ \hline
				$\lambda_{\{1, g^2\}}$ & $0$ & $0$ & $0$ &  $0$ & $0$ & $0$ & $0$ & $0$ & $0$ & $0$ & $0$ & $0$ \\ \hline
			\end{tabular}\vspace{0.3cm}
		%	\end{table}	
	%	\begin{table}[!ht]
			\caption{Partial actions of $\mathbf{H_{4}}$ on $\k$}\label{tab:H_4}\vspace{-0.45cm}
			\begin{tabular}{ | c | c | c | c | c | c | c | c | c | c | c | c | c | c | c | c | c } \hline
				& $x$ & $gx$ & $g^2x$ & $g^3x$ & $y$ & $gy$ & $g^2y$ & $g^3y$ & $xy$ &$gxy$ & $g^2xy$ & $g^3xy$ \\ \hline
				$\lambda_G=\varepsilon$  & $0$ & $0$ & $0$ & $0$ & $0$ & $0$ & $0$ & $0$ & $0$ & $0$ & $0$ & $0$ \\ \hline
				$\lambda_{\{1,g^2\}}$ & $\alpha$ & $\alpha$ & $\alpha$ &  $\alpha$ & $\beta$ & $\beta$ & $\beta$ & $\beta$ & $0$ & $0$ & $0$ & $0$ \\ \hline
				$\lambda_{\{1\}}$ & $0$ & $0$ &  $0$ & $0$ & $0$ & $0$ & $0$ & $0$ & $0$ & $0$ & $0$ & $0$ \\ \hline
			\end{tabular}\vspace{0.3cm}
		%	\end{table}	
	%	\begin{table}[!ht]	
			\caption{Partial actions of $\mathbf{H_{5}}$ on $\k$ $(\gamma^2 =-1)$}\label{tab:H_5}\vspace{-0.45cm}
			\begin{tabular}{ | c | c | c | c | c | c | c | c | c | c | c | c | c | c | c | c | c | } \hline
				& $x$ & $gx$ & $g^2x$ & $g^3x$ & $y$ & $gy$ & $g^2y$ & $g^3y$ & $xy$ &$gxy$ & $g^2xy$ & $g^3xy$ \\ \hline
				$\lambda_G=\varepsilon$  & $0$ & $0$ & $0$ & $0$ & $0$ & $0$ & $0$ & $0$ & $0$ & $0$ & $0$ & $0$ \\ \hline
				$\lambda_{\{1, g^2\}}$ & $\alpha$ & $\alpha$ & $\alpha$ &  $\alpha$ & $\beta$ & $\beta$ & $\beta$ & $\beta$ & $0$ & $0$ & $0$ & $0$ \\ \hline
				$\lambda_{\{1\}}$ & $0$ & $0$ &  $0$ & $0$ & $\gamma$ & $0$ & $0$ & $\gamma$ & $0$ & $0$ & $0$ & $0$ \\ \hline
			\end{tabular}\vspace{0.3cm}
	%	\end{table}	
%	\begin{table}[!ht]	
			\caption{Partial actions of $\mathbf{H_{6}}$ on $\k$}\label{tab:H_6}\vspace{-0.45cm}
			\begin{tabular}{ | c | c | c | c | c | c | c | c | c | c | c | c | c | c | c | c | c | } \hline
				& $x$ & $gx$ & $g^2x$ & $g^3x$ & $y$ & $gy$ & $g^2y$ & $g^3y$ & $xy$ &$gxy$ & $g^2xy$ & $g^3xy$ \\ \hline
				$\lambda_G=\varepsilon$  & $0$ & $0$ & $0$ & $0$ & $0$ & $0$ & $0$ & $0$ & $0$ & $0$ & $0$ & $0$ \\ \hline
				$\lambda_{\{1, g^2\}}$ & $\alpha$ & $\alpha$ & $\alpha$ &  $\alpha$ & $\beta$ & $\beta$ & $\beta$ & $\beta$ & $0$ & $0$ & $0$ & $0$ \\ \hline
			\end{tabular}\vspace{0.3cm}
%	\end{table}	
%	\begin{table}[!ht]
			\caption{Partial actions of $\mathbf{H_{7}}$ on $\k$}\label{tab:H_7}\vspace{-0.45cm}
			\begin{tabular}{ | c | c | c | c | c | c | c | c | c | c | c | c | c | c | c | c | c | } \hline
				& $x$ & $gx$ & $g^2x$ & $g^3x$ & $y$ & $gy$ & $g^2y$ & $g^3y$ & $xy$ &$gxy$ & $g^2xy$ & $g^3xy$ \\ \hline
				$\lambda_G = \varepsilon$  & $0$ & $0$ & $0$ & $0$ & $0$ & $0$ & $0$ & $0$ & $0$ & $0$ & $0$ & $0$ \\ \hline
				$\lambda_{\{1, g^2\}}$ & $\alpha$ & $\alpha$ & $\alpha$ &  $\alpha$ & $\beta$ & $\beta$ & $\beta$ & $\beta$ & $0$ & $0$ & $0$ & $0$ \\ \hline
				$\lambda_{\{1\}}$ & $0$ & $0$ &  $0$ & $0$ & $\alpha$ & $\alpha$ & $0$ & $0$ & $0$ & $0$ & $0$ & $0$ \\ \hline
			\end{tabular}\vspace{0.3cm}
	%	\end{table}	
%	\begin{table}[!ht]	
			\caption{Partial actions of $\mathbf{H_{8}}$ on $\k$ $(\gamma^2=-1)$}\label{tab:H_8}\vspace{-0.45cm}
			\begin{tabular}{ | c | c | c | c | c | c | c | c | c | c | c | c | c | c | c | c | c | } \hline
				& $x$ & $gx$ & $g^2x$ & $g^3x$ & $y$ & $gy$ & $g^2y$ & $g^3y$ & $xy$ &$gxy$ & $g^2xy$ & $g^3xy$ \\ \hline
				$\lambda_G=\varepsilon$  & $0$ & $0$ & $0$ & $0$ & $0$ & $0$ & $0$ & $0$ & $0$ & $0$ & $0$ & $0$ \\ \hline
				$\lambda_{\{1, g^2\}}$ &  $\alpha$ & $\alpha$ & $\alpha$ &  $\alpha$ & $\beta$ & $\beta$ & $\beta$ & $\beta$ & $0$ & $0$ & $0$ & $0$ \\ \hline
				$\lambda_{\{1\}}$ & $0$ & $0$ &  $0$ & $0$ & $\gamma$ & $\gamma$ & $0$ & $0$ & $0$ & $0$ & $0$ & $0$ \\ \hline
			\end{tabular}\vspace{0.3cm}
%	\end{table}	
%	\begin{table}[!ht]	
			\caption{Partial actions of $\mathbf{H_{9}}$ on $\k$ $(\gamma^2 = \omega^2 =-1)$}\label{tab:H_9}\vspace{-0.4cm}
			\begin{tabular}{ | c | c | c | c | c | c | c | c | c | c | c | c | c | c | c | c | c | } \hline
				& $x$ & $gx$ & $g^2x$ & $g^3x$ & $y$ & $gy$ & $g^2y$ & $g^3y$ & $xy$ &$gxy$ & $g^2xy$ & $g^3xy$ \\ \hline
				$\lambda_G=\varepsilon$  & $0$ & $0$ & $0$ & $0$ & $0$ & $0$ & $0$ & $0$ & $0$ & $0$ & $0$ & $0$ \\ \hline
				$\lambda_{\{1,g^2\}}$ & $\alpha$ & $\alpha$ & $\alpha$ &  $\alpha$ & $\beta$ & $\beta$ & $\beta$ & $\beta$ & $0$ & $0$ & $0$ & $0$ \\ \hline
				$\lambda_{\{1\}}$ & $\gamma$ & $0$ &  $0$ & $\gamma$ & $\omega$ & $\omega$ & $0$ & $0$ & $0$ & $-\gamma \omega$ & $0$ & $\gamma \omega$ \\ \hline
			\end{tabular}\vspace{0.3cm}
%	\end{table}	
%	\begin{table}[!ht]
			\caption{Partial actions of $\mathbf{H_{10}}$ on $\k$}\label{tab:H_10}\vspace{-0.45cm}
			\begin{tabular}{ | c | c | c | c | c | c | c | c | c | c | c | c | c | c | c | c | c | } \hline
				& $x$ & $gx$ & $g^2x$ & $g^3x$ & $x^2$ & $gx^2$ & $g^2x^2$ & $g^3x^2$ & $x^3$ &$gx^3$ & $g^2x^3$ & $g^3x^3$ \\ \hline
				$\lambda_G\!=\!\varepsilon$ & $0$ & $0$ & $0$ & $0$ & $0$ & $0$ & $0$ & $0$ & $0$ & $0$ & $0$ & $0$ \\ \hline
				$\lambda_{\{1, g^2\}}$ & $0$ & $0$ & $0$ &  $0$ & $0$ & $0$ & $0$ & $0$ & $0$ & $0$ & $0$ & $0$ \\ \hline
				$\lambda_{\{1\}}$ & $\alpha$ & $0$ &  $0$ & $-q \alpha$ & $\alpha^2$ & $0$ & $-q \alpha^2$ & $(1\!-\!q)\alpha^2$ & $\alpha^3$ & $\alpha^3$ & $\alpha^3$ & $\alpha^3$ \\ \hline
			\end{tabular}\vspace{0.3cm}
	\end{table}	
	\begin{table}[!ht]
			\caption{Partial actions of $\mathbf{H_{11}}$ on $\k$}\label{tab:H_11}\vspace{-0.4cm}
			\begin{tabular}{ | c | c | c | c | c | c | c | c | c | c | c | c | c | c | c | c | c | } \hline
				& $x$ & $gx$ & $g^2x$ & $g^3x$ & $x^2$ & $gx^2$ & $g^2x^2$ & $g^3x^2$ & $x^3$ &$gx^3$ & $g^2x^3$ & $g^3x^3$ \\ \hline
				$\lambda_G\!=\!\varepsilon$ & $0$ & $0$ & $0$ & $0$ & $0$ & $0$ & $0$ & $0$ & $0$ & $0$ & $0$ & $0$ \\ \hline
				$\lambda_{\{1, g^2\}}$ & $0$ & $0$ & $0$ &  $0$ & $0$ & $0$ & $0$ & $0$ & $0$ & $0$ & $0$ & $0$ \\ \hline
				$\lambda_{\{1\}}$ & $\alpha$ & $0$ &  $0$ & $q \alpha$ & $\alpha^2$ & $0$ & $q \alpha^2$ & $(1\!+\!q)\alpha^2$ & $\alpha^3$ & $\alpha^3$ & $\alpha^3$ & $\alpha^3$ \\ \hline
			\end{tabular}\vspace{0.3cm}
	%	\end{table}	
%	\begin{table}[!ht]	
			\caption{Partial actions of $\mathbf{H_{12}}$ on $\k$}\label{tab:H_12}\vspace{-0.4cm}
			\begin{tabular}{ | c | c | c | c | c | c | c | c | c | c | c | c | c | c | c | c | c | }\hline
				& $x$ & $hx$  & $y$ & $hy$ & $gx$ & $ghx$  & $gy$ & $ghy$ & $xy$ & $hxy$  & $gxy$ & $ghxy$ \\ \hline
				$\lambda_G=\varepsilon$  & $0$ & $0$ & $0$ & $0$ & $0$ & $0$ & $0$ & $0$ & $0$ & $0$ & $0$ & $0$ \\ \hline
				$\lambda_{\{1,h\}}$ & $\alpha$ & $\alpha$ & $\beta$ &  $\beta$ & $\alpha$ & $\alpha$ & $\beta$ &  $\beta$ & $0$ & $0$ & $0$ & $0$ \\ \hline
				$\lambda_{\{1,g\}}$ & $0$ & $0$ &  $0$ & $0$ & $0$ & $0$ & $0$ & $0$ & $0$ & $0$ & $0$ & $0$ \\ \hline
				$\lambda_{\{1,gh\}}$ & $0$ & $0$ & $0$ & $0$ & $0$ & $0$ & $0$ & $0$ & $0$ & $0$ & $0$ & $0$ \\ \hline
				$\lambda_{\{1\}}$ & $\alpha$ & $0$ & $\beta$ & $0$ & $\alpha$ & $0$ & $\beta$ & $0$ & $0$ & $0$ & $0$ & $0$ \\ \hline
			\end{tabular}\vspace{0.3cm}
		%	\end{table}	
	%	\begin{table}[!ht]
			\caption{Partial actions of $\mathbf{H_{13}}$ on $\k$}\label{tab:H_13}\vspace{-0.4cm}
			\begin{tabular}{ | c | c | c | c | c | c | c | c | c | c | c | c | c | c | c | c | c | } \hline
				& $x$ & $gx$ & $hx$ & $ghx$ & $y$ & $gy$ & $hy$ & $ghy$ & $xy$ &$gxy$ & $hxy$ & $ghxy$ \\ \hline
				$\lambda_G=\varepsilon$  & $0$ & $0$ & $0$ & $0$ & $0$ & $0$ & $0$ & $0$ & $0$ & $0$ & $0$ & $0$ \\ \hline
				$\lambda_{\{1,g\}}$ & $0$ & $0$ & $0$ &  $0$ & $0$ & $0$ & $0$ & $0$ & $0$ & $0$ & $0$ & $0$ \\ \hline
				$\lambda_{\{1,h\}}$ & $\alpha$ & $\alpha$ &  $\alpha$ & $\alpha$ & $0$ & $0$ & $0$ & $0$ & $0$ & $0$ & $0$ & $0$ \\ \hline
				$\lambda_{\{1,gh\}}$ & $0$ & $0$ & $0$ & $0$ & $\alpha$ & $\alpha$ & $\alpha$ & $\alpha$ & $0$ & $0$ & $0$ & $0$ \\ \hline
				$\lambda_{\{1\}}$ & $\alpha$ & $\alpha$ & $0$ & $0$ & $\beta$ & $\beta$ & $0$ & $0$ & $0$ & $0$ & $0$ & $0$ \\ \hline
			\end{tabular}\vspace{0.3cm}
	%	\end{table}		
	%	\begin{table}[!ht]
			\caption{Partial actions of $\mathbf{H_{14}}$ on $\k$}\label{tab:H_14}\vspace{-0.4cm}
			\begin{tabular}{| c | c | c | c | c | c | c | c | c | c | c | c | c | c | c | c | c |} \hline
				& $y$ & $hy $  & $x$ & $gx$	& $gy$ & $ghy $  & $hx$ & $ghx$ & $xy$ & $gxy $  & $hxy$ & $ghxy$   \\ \hline
				$\lambda_G=\varepsilon$  & $0$ & $0$ & $0$ & $0$ & $0$ & $0$ & $0$ & $0$ & $0$ & $0$ & $0$ & $0$ \\ \hline
				$\lambda_{\{1,h\}}$ & $0$ & $0$ & $\alpha$ &  $\alpha$ & $0$ & $0$ & $\alpha$ &   $\alpha$ & $0$ & $0$ & $0$ &  $0$\\ \hline
				$\lambda_{\{1,g\}}$ & $\alpha$ & $\alpha$ &  $0$ & $0$  & $\alpha$ & $\alpha$ &  $0$ & $0$  & $0$ & $0$ &  $0$ & $0$ \\ \hline
				$\lambda_{\{1,gh\}}$ & $0$ & $0$ & $0$ & $0$ & $0$ & $0$ &  $0$ & $0$ & $0$ & $0$ &  $0$ & $0$ \\ \hline
				$\lambda_{\{1\}}$ & $\beta$ & $\beta$ & $\alpha$ & $\alpha$ & $0$ & $0$ &  $0$ & $0$  & $\alpha \beta$ & $\alpha \beta$  & $\alpha \beta$ & $\alpha \beta$ \\ \hline
			\end{tabular}\vspace{0.3cm}
	%	\end{table}		
	%	\begin{table}[!ht]
			\caption{Partial actions of $\mathbf{H_{15}}$ on $\k$ $(\sigma \delta = 0)$}\label{tab:H_15}\vspace{-0.4cm}
			\begin{tabular}{ | c | c | c | c | c | c | c | c | c | c | c | c | c | c | c | c | c | } \hline
				& $x$ & $gx$ & $hx$ & $ghx$ & $y$ & $gy$ & $hy$ & $ghy$ & $xy$ &$gxy$ & $hxy$ & $ghxy$ \\ \hline
				$\lambda_G=\varepsilon$  & $0$ & $0$ & $0$ & $0$ & $0$ & $0$ & $0$ & $0$ & $0$ & $0$ & $0$ & $0$ \\ \hline
				$\lambda_{\{1, g\}}$ & $0$ & $0$ & $0$ &  $0$ & $0$ & $0$ & $0$ & $0$ & $0$ & $0$ & $0$ & $0$ \\ \hline
				$\lambda_{\{1, h\}}$ & $0$ & $0$ &  $0$ & $0$ & $0$ & $0$ & $0$ & $0$ & $0$ & $0$ & $0$ & $0$ \\ \hline
				$\lambda_{\{1, gh\}}$ & $\alpha$ & $\alpha$ & $\alpha$ & $\alpha$ & $\beta$ & $\beta$ & $\beta$ & $\beta$ & $0$ & $0$ & $0$ & $0$ \\ \hline
				$\lambda_{\{1 \}}$ & $\sigma$ & $\sigma$ & $0$ & $0$ & $\delta$ & $0$ & $\delta$ & $0$ & $0$ & $0$ & $0$ & $0$ \\  \hline
			\end{tabular}\vspace{0.3cm}
%		\end{table}		
%		\begin{table}[!ht]
			\caption{Partial actions of $\mathbf{H_{16}}$ on $\k$ $\left(\Omega \zeta = - \frac{1}{2}\right)$}\label{tab:H_16}\vspace{-0.4cm}
			\begin{tabular}{ | c | c | c | c | c | c | c | c | c | c | c | c | c | c | c | c | c | } \hline
				& $x$ & $gx$ & $hx$ & $ghx$ & $y$ & $gy$ & $hy$ & $ghy$ & $xy$ &$gxy$ & $hxy$ & $ghxy$ \\ \hline
				$\lambda_G=\varepsilon$  & $0$ & $0$ & $0$ & $0$ & $0$ & $0$ & $0$ & $0$ & $0$ & $0$ & $0$ & $0$ \\ \hline
				$\lambda_{\{1, gh\}}$ & $\alpha$ & $\alpha$ & $\alpha$ & $\alpha$ & $\beta$ & $\beta$ & $\beta$ & $\beta$ & $0$ & $0$ & $0$ & $0$ \\ \hline
				$\lambda_{\{1 \}}$ & $\Omega$ & $\Omega$ & $0$ & $0$ & $\zeta$ & $0$ & $\zeta$ & $0$ & $- \frac{1}{2}$ & $- \frac{1}{2}$ & $\frac{1}{2}$ & $\frac{1}{2}$ \\  \hline
			\end{tabular}\vspace{0.3cm}
%	\end{table}
%	\begin{table}[!ht]
			\caption{Partial actions of $\mathbf{H_{17}}$ on $\k$}\label{tab:H_17}\vspace{-0.4cm}
			\begin{tabular}{ | c | c | c | c | c | c | c | c | c | c | c | c | c | c | c | c | c | } \hline
				& $x$ & $ax$ & $bx$ & $abx$ & $gx$ & $agx$ & $bgx$ & $abgx$  \\ \hline
				$\lambda_G=\varepsilon$  & $0$ & $0$ & $0$ & $0$ & $0$ & $0$ & $0$ & $0$ \\ \hline
				$\lambda_{\{1, a, g, ag \}}$ & $0$ & $0$ & $0$ & $0$ & $0$ & $0$ & $0$ & $0$ \\ \hline
				$\lambda_{\{1, b, g, bg \}}$ & $0$ & $0$ & $0$ & $0$ & $0$ & $0$ & $0$ & $0$ \\ \hline
				$\lambda_{\{1, ab, g, abg \}}$ & $0$ & $0$ & $0$ & $0$ & $0$ & $0$ & $0$ & $0$ \\ \hline
				$\lambda_{\{1, a, gb, agb \}}$ & $0$ & $0$ & $0$ & $0$ & $0$ & $0$ & $0$ & $0$ \\ \hline
				$\lambda_{\{1, b, ag, abg \}}$ & $0$ & $0$ & $0$ & $0$ & $0$ & $0$ & $0$ & $0$ \\ \hline
				$\lambda_{\{1, ab, ag, bg \}}$ & $0$ & $0$ & $0$ & $0$ & $0$ & $0$ & $0$ & $0$ \\ \hline
				$\lambda_{\{1, a, b, ab \}}$ &  $\alpha$ & $\alpha$ & $\alpha$ & $\alpha$ & $\alpha$ & $\alpha$ & $\alpha$ & $\alpha$ \\ \hline
				$\lambda_{\{1, g\}}$ & $0$ & $0$ & $0$ & $0$ & $0$ & $0$ & $0$ & $0$ \\ \hline
				$\lambda_{\{1, ag \}}$ & $0$ & $0$ & $0$ & $0$ & $0$ & $0$ & $0$ & $0$ \\ \hline
				$\lambda_{\{1, bg \}}$ & $0$ & $0$ & $0$ & $0$ & $0$ & $0$ & $0$ & $0$ \\ \hline
				$\lambda_{\{1, abg \}}$ & $0$ & $0$ & $0$ & $0$ & $0$ & $0$ & $0$ & $0$ \\ \hline
				$\lambda_{\{1, a\}}$ & $\alpha$ & $\alpha$ & $0$ & $0$ & $\alpha$ & $\alpha$ & $0$ & $0$ \\ \hline
				$\lambda_{\{1, b\}}$ & $\alpha$ & $0$  & $\alpha$ & $0$ & $\alpha$ &  $0$ & $\alpha$ & $0$ \\ \hline
				$\lambda_{\{1, ab\}}$ & $\alpha$ &  $0$ & $0$ & $\alpha$ & $\alpha$  & $0$ & $0$ & $\alpha$ \\ \hline
				$\lambda_{\{1 \}}$ & $\alpha$ & $0$ & $0$ & $0$ & $\alpha$ & $0$ & $0$ & $0$ \\ \hline
			\end{tabular}\vspace{0.3cm}
		\end{table}
		\begin{table}[!ht]
			\caption{Partial actions of $\mathbf{H_{18}}$ on $\k$}\label{tab:H_18}\vspace{-0.4cm}
			\begin{tabular}{ | c | c | c | c | c | c | c | c | c | c | c | c | c | c | c | c | c | } \hline
				& $x$ & $gx$ & $g^2x$ & $g^3x$ & $g^4x$ &$g^5x$ & $g^6x$ & $g^7x$ \\ \hline
				$\lambda_G=\varepsilon$  & $0$ & $0$ & $0$ & $0$ & $0$ & $0$ & $0$ & $0$ \\ \hline
				$\lambda_{\{1, g^2, g^4, g^6\}}$ & $\alpha$ & $\alpha$ & $\alpha$ & $\alpha$ & $\alpha$ & $\alpha$ & $\alpha$ & $\alpha$ \\ \hline
				$\lambda_{\{1, g^4\}}$ & $0$ & $0$ & $0$ & $0$ & $0$ & $0$ & $0$ & $0$ \\ \hline
				$\lambda_{\{1 \}}$ & $0$ & $0$ & $0$ & $0$ & $0$ & $0$ & $0$ & $0$ \\ \hline
			\end{tabular}\vspace{0.3cm}
%		\end{table}		
%		\begin{table}[!ht]
			\caption{Partial actions of $\mathbf{H_{19}}$ on $\k$}\label{tab:H_19}\vspace{-0.4cm}
			\begin{tabular}{ | c | c | c | c | c | c | c | c | c | c | c | c | c | c | c | c | c | } \hline
				& $x$ & $gx$ & $g^2x$ & $g^3x$ & $g^4x$ &$g^5x$ & $g^6x$ & $g^7x$ \\ \hline
				$\lambda_G=\varepsilon$  & $0$ & $0$ & $0$ & $0$ & $0$ & $0$ & $0$ & $0$ \\ \hline
				$\lambda_{\{1, g^2, g^4, g^6\}}$ & $0$ & $0$ & $0$ & $0$ & $0$ & $0$ & $0$ & $0$  \\ \hline
				$\lambda_{\{1, g^4\}}$ & $0$ & $0$ & $0$ & $0$ & $0$ & $0$ & $0$ & $0$ \\ \hline
				$\lambda_{\{1 \}}$ & $\alpha$ & $0$ & $0$ & $0$ & $\alpha$ & $0$ & $0$ & $0$ \\ \hline
			\end{tabular}\vspace{0.3cm}
%		\end{table}	
%		\begin{table}[!ht]
			\caption{Partial actions of $\mathbf{H_{20}}$ on $\k$}\label{tab:H_20}\vspace{-0.4cm}
			\begin{tabular}{ | c | c | c | c | c | c | c | c | c | c | c | c | c | c | c | c | c | } \hline
				& $x$ & $gx$ & $g^2x$ & $g^3x$ & $g^4x$ &$g^5x$ & $g^6x$ & $g^7x$ \\ \hline
				$\lambda_G=\varepsilon$  & $0$ & $0$ & $0$ & $0$ & $0$ & $0$ & $0$ & $0$ \\ \hline
				$\lambda_{\{1, g^2, g^4, g^6\}}$ & $0$ & $0$ & $0$ & $0$ & $0$ & $0$ & $0$ & $0$  \\ \hline
				$\lambda_{\{1, g^4\}}$ & $\alpha$ & $0$ & $\alpha$ & $0$ & $\alpha$ & $0$ & $\alpha$ & $0$ \\ \hline
				$\lambda_{\{1 \}}$ & $0$ & $0$ & $0$ & $0$ & $0$ & $0$ & $0$ & $0$ \\ \hline
			\end{tabular}\vspace{0.3cm}
	%	\end{table}		
	%	\begin{table}[!ht]
			\caption{Partial actions of $\mathbf{H_{21}}$ on $\k$}\label{tab:H_21}\vspace{-0.4cm}
			\begin{tabular}{ | c | c | c | c | c | c | c | c | c | c | c | c | c | c | c | c | c | } \hline
				& $x$ & $gx$ & $g^2x$ & $g^3x$ & $g^4x$ &$g^5x$ & $g^6x$ & $g^7x$ \\ \hline
				$\lambda_G=\varepsilon$  & $0$ & $0$ & $0$ & $0$ & $0$ & $0$ & $0$ & $0$ \\ \hline
				$\lambda_{\{1, g^2, g^4, g^6\}}$ & $0$ & $0$ & $0$ & $0$ & $0$ & $0$ & $0$ & $0$  \\ \hline
				$\lambda_{\{1, g^4\}}$ & $\alpha$ & $0$ & $\alpha$ & $0$ & $\alpha$ & $0$ & $\alpha$ & $0$ \\ \hline
				$\lambda_{\{1 \}}$ & $0$ & $0$ & $0$ & $0$ & $0$ & $0$ & $0$ & $0$ \\ \hline
			\end{tabular}\vspace{0.3cm}
	%	\end{table}	
	%	\begin{table}[!ht]
			\caption{Partial actions of $\mathbf{H_{22}}$ on $\k$ $(\gamma^2=-1)$}\label{tab:H_22}\vspace{-0.4cm}
			\begin{tabular}{ | c | c | c | c | c | c | c | c | c | c | c | c | c | c | c | c | c | } \hline
				& $x$ & $gx$ & $g^2x$ & $g^3x$ & $g^4x$ &$g^5x$ & $g^6x$ & $g^7x$ \\ \hline
				$\lambda_G=\varepsilon$  & $0$ & $0$ & $0$ & $0$ & $0$ & $0$ & $0$ & $0$ \\ \hline
				$\lambda_{\{1, g^2, g^4, g^6\}}$ & $\alpha$ & $\alpha$ & $\alpha$ & $\alpha$ & $\alpha$ & $\alpha$ & $\alpha$ & $\alpha$  \\ \hline
				$\lambda_{\{1, g^4\}}$ & $\gamma$ & $0$ & $0$ & $\gamma$ & $\gamma$ & $0$ & $0$ & $\gamma$ \\ \hline
				$\lambda_{\{1 \}}$ & $\gamma$ & $0$ & $0$ & $0$ & $0$ & $0$ & $0$ & $\gamma$ \\ \hline
			\end{tabular}\vspace{0.3cm}
%		\end{table}		
%		\begin{table}[!ht]
			\caption{Partial actions of $\mathbf{H_{23}}$ on $\k$}\label{tab:H_23}\vspace{-0.4cm}
			\begin{tabular}{ | c | c | c | c | c | c | c | c | c | c | c | c | c | c | c | c | c | } \hline
				& $x$ & $gx$ & $g^2x$ & $g^3x$ & $hx$ &$ghx$ & $g^2hx$ & $g^3hx$ \\ \hline
				$\lambda_G=\varepsilon$  & $0$ & $0$ & $0$ & $0$ & $0$ & $0$ & $0$ & $0$ \\ \hline
				$\lambda_{\{1, g, g^2, g^3\}}$ & $0$ & $0$ & $0$ & $0$ & $0$ & $0$ & $0$ & $0$ \\ \hline
				$\lambda_{\{1, g^2, h, g^2h\}}$ & $\alpha$ & $\alpha$ & $\alpha$ & $\alpha$ & $\alpha$ & $\alpha$ & $\alpha$ & $\alpha$ \\ \hline
				$\lambda_{\{1, g^2h\}}$ & $0$ & $0$ & $0$ & $0$ & $0$ & $0$ & $0$ & $0$ \\ \hline
				$\lambda_{\{1, h\}}$ & $0$ & $0$ & $0$ & $0$ & $0$ & $0$ & $0$ & $0$ \\ \hline
				$\lambda_{\{1, g^2, gh, g^3h\}}$ & $0$ & $0$ & $0$ & $0$ & $0$ & $0$ & $0$ & $0$ \\ \hline
				$\lambda_{\{1, g^2\}}$ & $\alpha$ & $\alpha$ & $\alpha$ & $\alpha$ & $0$ & $0$ & $0$ & $0$ \\ \hline
				$\lambda_{\{1 \}}$ & $0$ & $0$ & $0$ & $0$ & $0$ & $0$ & $0$ & $0$ \\ \hline
			\end{tabular}\vspace{0.3cm}
%		\end{table}		
%		\begin{table}[!ht]
			\caption{Partial actions of $\mathbf{H_{24}}$ on $\k$}\label{tab:H_24}\vspace{-0.4cm}
			\begin{tabular}{ | c | c | c | c | c | c | c | c | c | c | c | c | c | c | c | c | c | } \hline
				& $x$ & $gx$ & $g^2x$ & $g^3x$ & $hx$ &$ghx$ & $g^2hx$ & $g^3hx$ \\ \hline
				$\lambda_G=\varepsilon$  & $0$ & $0$ & $0$ & $0$ & $0$ & $0$ & $0$ & $0$ \\ \hline
				$\lambda_{\{1, g, g^2, g^3\}}$ & $\alpha$ & $\alpha$ & $\alpha$ & $\alpha$ & $\alpha$ & $\alpha$ & $\alpha$ & $\alpha$ \\ \hline
				$\lambda_{\{1, g^2, h, g^2h\}}$ & $0$ & $0$ & $0$ & $0$ & $0$ & $0$ & $0$ & $0$ \\ \hline
				$\lambda_{\{1, g^2h\}}$ & $0$ & $0$ & $0$ & $0$ & $0$ & $0$ & $0$ & $0$ \\ \hline
				$\lambda_{\{1, h\}}$ & $0$ & $0$ & $0$ & $0$ & $0$ & $0$ & $0$ & $0$ \\ \hline
				$\lambda_{\{1, g^2, gh, g^3h\}}$ & $0$ & $0$ & $0$ & $0$ & $0$ & $0$ & $0$ & $0$ \\ \hline
				$\lambda_{\{1, g^2\}}$ & $\alpha$ & $0$ & $\alpha$ & $0$ & $0$ & $\alpha$ & $0$ & $\alpha$ \\ \hline
				$\lambda_{\{1 \}}$ & $0$ & $0$ & $0$ & $0$ & $0$ & $0$ & $0$ & $0$ \\ \hline
			\end{tabular}\vspace{0.3cm}
		\end{table}		
		\begin{table}[!ht]
			\caption{Partial actions of $\mathbf{H_{25}}$ on $\k$}\label{tab:H_25}\vspace{-0.4cm}
			\begin{tabular}{ | c | c | c | c | c | c | c | c | c | c | c | c | c | c | c | c | c | } \hline
				& $x$ & $gx$ & $g^2x$ & $g^3x$ & $hx$ &$ghx$ & $g^2hx$ & $g^3hx$ \\ \hline
				$\lambda_G=\varepsilon$  & $0$ & $0$ & $0$ & $0$ & $0$ & $0$ & $0$ & $0$ \\ \hline
				$\lambda_{\{1, g, g^2, g^3\}}$ & $0$ & $0$ & $0$ & $0$ & $0$ & $0$ & $0$ & $0$ \\ \hline
				$\lambda_{\{1, g^2, h, g^2h\}}$ & $0$ & $0$ & $0$ & $0$ & $0$ & $0$ & $0$ & $0$ \\ \hline
				$\lambda_{\{1, g^2h\}}$ & $0$ & $0$ & $0$ & $0$ & $0$ & $0$ & $0$ & $0$ \\ \hline
				$\lambda_{\{1, h\}}$ & $\alpha$ & $0$ & $\alpha$ & $0$ & $\alpha$ & $0$ & $\alpha$ & $0$ \\ \hline
				$\lambda_{\{1, g^2, gh, g^3h\}}$ & $0$ & $0$ & $0$ & $0$ & $0$ & $0$ & $0$ & $0$ \\ \hline
				$\lambda_{\{1, g^2\}}$ & $0$ & $0$ & $0$ & $0$ & $0$ & $0$ & $0$ & $0$ \\ \hline
				$\lambda_{\{1 \}}$ & $\alpha$ & $0$ & $\alpha$ & $0$ & $0$ & $0$ & $0$ & $0$ \\ \hline
			\end{tabular}\vspace{0.3cm}
	%	\end{table}	
	%	\begin{table}[!ht]
			\caption{Partial actions of $\mathbf{H_{26}}$ on $\k$}\label{tab:H_26}\vspace{-0.4cm}
			\begin{tabular}{ | c | c | c | c | c | c | c | c | c | c | c | c | c | c | c | c | c | } \hline
				& $x$ & $gx$ & $g^2x$ & $g^3x$ & $hx$ &$ghx$ & $g^2hx$ & $g^3hx$ \\ \hline
				$\lambda_G=\varepsilon$  & $0$ & $0$ & $0$ & $0$ & $0$ & $0$ & $0$ & $0$ \\ \hline
				$\lambda_{\{1, g, g^2, g^3\}}$ & $\alpha$ & $\alpha$ & $\alpha$ & $\alpha$ & $\alpha$ & $\alpha$ & $\alpha$ & $\alpha$ \\ \hline
				$\lambda_{\{1, g^2, h, g^2h\}}$ & $0$ & $0$ & $0$ & $0$ & $0$ & $0$ & $0$ & $0$ \\ \hline
				$\lambda_{\{1, g^2h\}}$ & $0$ & $0$ & $0$ & $0$ & $0$ & $0$ & $0$ & $0$ \\ \hline
				$\lambda_{\{1, h\}}$ & $0$ & $0$ & $0$ & $0$ & $0$ & $0$ & $0$ & $0$ \\ \hline
				$\lambda_{\{1, g^2, gh, g^3h\}}$ & $0$ & $0$ & $0$ & $0$ & $0$ & $0$ & $0$ & $0$ \\ \hline
				$\lambda_{\{1, g^2\}}$ & $\alpha$ & $0$ & $\alpha$ & $0$ & $\alpha$ & $0$ & $\alpha$ & $0$ \\ \hline
				$\lambda_{\{1 \}}$ & $\alpha$ & $0$ & $0$ & $0$ & $\alpha$ & $0$ & $0$ & $0$ \\ \hline
			\end{tabular}\vspace{0.3cm}
%		\end{table}
%		\begin{table}[!ht]
			\caption{Partial actions of $\mathbf{H_{27}}$ on $\k$}\label{tab:H_27}\vspace{-0.4cm}
			\begin{tabular}{ | c | c | c | c | c | c | c | c | c | c | c | c | c | c | c | c | c | } \hline
				& $x$ & $gx$ & $g^2x$ & $g^3x$ & $hx$ &$ghx$ & $g^2hx$ & $g^3hx$ \\ \hline
				$\lambda_G=\varepsilon$  & $0$ & $0$ & $0$ & $0$ & $0$ & $0$ & $0$ & $0$ \\ \hline
				$\lambda_{\{1, g, g^2, g^3\}}$ & $0$ & $0$ & $0$ & $0$ & $0$ & $0$ & $0$ & $0$ \\ \hline
				$\lambda_{\{1, g^2, h, g^2h\}}$ & $0$ & $0$ & $0$ & $0$ & $0$ & $0$ & $0$ & $0$ \\ \hline
				$\lambda_{\{1, g^2h\}}$ & $0$ & $0$ & $0$ & $0$ & $0$ & $0$ & $0$ & $0$ \\ \hline
				$\lambda_{\{1, h\}}$ & $\alpha$ & $0$ & $\alpha$ & $0$ & $\alpha$ & $0$ & $\alpha$ & $0$ \\ \hline
				$\lambda_{\{1, g^2, gh, g^3h\}}$ & $0$ & $0$ & $0$ & $0$ & $0$ & $0$ & $0$ & $0$ \\ \hline
				$\lambda_{\{1, g^2\}}$ & $0$ & $0$ & $0$ & $0$ & $0$ & $0$ & $0$ & $0$ \\ \hline
				$\lambda_{\{1 \}}$ & $\alpha$ & $0$ & $0$ & $0$ & $0$ & $0$ & $\alpha$ & $0$ \\ \hline
			\end{tabular}\vspace{0.3cm}
%		\end{table}	
%		\begin{table}[!ht]
			\caption{Partial actions of $\mathbf{H_{28}}$ on $\k$ $(\gamma^2=-1)$}\label{tab:H_28}\vspace{-0.4cm}
			\begin{tabular}{ | c | c | c | c | c | c | c | c | c | c | c | c | c | c | c | c | c | } \hline
				& $x$ & $gx$ & $g^2x$ & $g^3x$ & $hx$ &$ghx$ & $g^2hx$ & $g^3hx$ \\ \hline
				$\lambda_G=\varepsilon$  & $0$ & $0$ & $0$ & $0$ & $0$ & $0$ & $0$ & $0$ \\ \hline
				$\lambda_{\{1, g, g^2, g^3\}}$ & $0$ & $0$ & $0$ & $0$ & $0$ & $0$ & $0$ & $0$ \\ \hline
				$\lambda_{\{1, g^2, h, g^2h\}}$  & $\alpha$ & $\alpha$ & $\alpha$ & $\alpha$  & $\alpha$ & $\alpha$ & $\alpha$ & $\alpha$  \\ \hline
				$\lambda_{\{1, g^2h\}}$ & $\gamma$ & $0$ & $0$ & $\gamma$ & $0$ & $\gamma$ & $\gamma$ & $0$ \\ \hline
				$\lambda_{\{1, h\}}$ & $\gamma$ & $0$ & $0$ & $\gamma$ & $\gamma$ & $0$ & $0$ & $\gamma$ \\ \hline
				$\lambda_{\{1, g^2, gh, g^3h\}}$ & $0$ & $0$ & $0$ & $0$ & $0$ & $0$ & $0$ & $0$ \\ \hline
				$\lambda_{\{1, g^2\}}$ & $\alpha$ & $\alpha$ & $\alpha$ & $\alpha$ & $0$ & $0$ & $0$ & $0$ \\ \hline
				$\lambda_{\{1 \}}$ & $\gamma$ & $0$ & $0$ & $\gamma$ & $0$ & $0$ & $0$ & $0$ \\ \hline
			\end{tabular}\vspace{0.3cm}
%		\end{table}
%	\end{center}
%\begin{center}
%		\begin{table}[!ht]
			\caption{Partial actions of $\mathbf{H_{29}}$ on $\k$ $(\gamma^2=-1)$}\label{tab:H_29}\vspace{-0.4cm}
			\begin{tabular}{ | c | c | c | c | c | c | c | c | c | c | c | c | c | c | c | c | c | } \hline
				& $x$ & $gx$ & $g^2x$ & $g^3x$ & $hx$ &$ghx$ & $g^2hx$ & $g^3hx$ \\ \hline
				$\lambda_G=\varepsilon$  & $0$ & $0$ & $0$ & $0$ & $0$ & $0$ & $0$ & $0$ \\ \hline
				$\lambda_{\{1, g, g^2, g^3\}}$ & $\alpha$ & $\alpha$ & $\alpha$ & $\alpha$  & $\alpha$ & $\alpha$ & $\alpha$ & $\alpha$   \\ \hline
				$\lambda_{\{1, g^2, h, g^2h\}}$  & $0$ & $0$ & $0$ & $0$ & $0$ & $0$ & $0$ & $0$  \\ \hline
				$\lambda_{\{1, g^2, gh, g^3h\}}$ & $0$ & $0$ & $0$ & $0$ & $0$ & $0$ & $0$ & $0$ \\ \hline
				$\lambda_{\{1, g^2\}}$ & $\alpha$ & $0$ & $\alpha$ & $0$ & $0$ & $\alpha$ & $0$ & $\alpha$ \\ \hline
				$\lambda_{\{1 \}}$ & $\gamma$ & $0$ & $0$ & $0$ & $0$ & $0$ & $0$ & $\gamma$ \\ \hline
			\end{tabular}
		\end{table}	
	\end{center}
}
\end{prop}

\section{$\lambda$-Hopf algebras}\label{sec:lda_Hopf_Alg}
In this section, the notion of a $\lambda$-Hopf algebra is introduced.
Some results and examples are developed.
In particular, we relate $\lambda$-Hopf algebras with a special type of partial matched pairs and we use the new concept to produce new examples of them.
Subsection \ref{Sec: Taft como lambda Hopf} is dedicated to present Taft's algebra as a $\lambda$-Hopf algebra.
Finally, the $\lambda$-Hopf algebras produced with the partial actions of pointed non-semisimple Hopf algebras with dimension 8, 16 on their base field are calculated in the last subsection.

\subsection{Definition and Properties}
Recall from Definition \ref{algebra_smash} that, given a partial action $\lambda: H \longrightarrow \k$, the corresponding partial smash product algebra $\underline{ \k \# H}$ is the algebra generated by the set $\left\{\underline{ 1_\k \# h} \, | \, h \in H \right\} = \left\{ 1_\k \# \lambda(h_1)h_2 \, | \, h \in H \right\}$.
Moreover $\underline{ 1_\k \# h} =  1_\k \# \lambda(h_1)h_2 $ and $ \underline{ 1_\k \# h} \, \, \underline{ 1_\k \# k} = \underline{ 1_\k \# \lambda(h_1)h_2k}$ for all $h, k \in H$.
Then, $\varphi: \underline{ \k \# H} \longrightarrow H$ given by $\varphi \left(1_\k \# \lambda(h_1)h_2\right) = \lambda(h_1)h_2$ is an injective homomorphism of algebras.
We denote by $H_\lambda$ the subalgebra of $H$ given by the image of $\varphi$, that is, $H_\lambda = \varphi\left(\underline{ \k \# H}\right) \subseteq H$.
Hence, $H_\lambda = \{\lambda(h_1)h_2 \, | \, h \in H \}$ and so $\underline{ \k \# H} \simeq H_\lambda $ is a subalgebra of $H$.
However, it does not always occur that this subalgebra is a Hopf subalgebra, neither a subbialgebra nor a subcoalgebra of $H$.
We exhibit two examples to illustrate this situation:

\begin{exa}\label{grupo_eh_Hlda}
	Let $G$ be a group and $\lambda_N$ as in Example \ref{lda_kg}.
	Then, $(\k G)_{\lambda_{N}} \simeq \k N$ and consequently a Hopf subalgebra of $\k G$.
\end{exa}

\begin{exa}\label{Sweedler's_nao_Hlda}
	Assume $char(\k) \neq 2$ and let $\mathbb{H}_4$ be the Sweedler's Hopf algebra.
	Consider the partial action $\lambda_\alpha$ as in Example \ref{lda_sweedler}, \emph{i.e.}, $\lambda_\alpha: \mathbb{H}_4 \longrightarrow \k$, $\lambda_\alpha(1)=1_\k,$ $\lambda_\alpha(g)=0$ and $\lambda_\alpha(x) = \lambda_\alpha(gx) = \alpha$, where $\alpha \in \k$.
	
	The subalgebra $(\mathbb{H}_4)_{\lambda_\alpha} = \k \{1\} \oplus \k \{ \alpha g + gx\}$
	is not a subcoalgebra of $\mathbb{H}_4$, since
	$$\Delta(\alpha g + gx) = (\alpha g + gx) \o g + 1 \o gx \notin (\mathbb{H}_4)_{\lambda_\alpha} \o (\mathbb{H}_4)_{\lambda_\alpha}.$$
\end{exa}

The previous examples show that for a given Hopf algebra $L$ and a partial action $\lambda \in L^*$, $L_\lambda$ is not always a Hopf subalgebra of $L$.
Thus, we are interested in deciding when a Hopf algebra $H$ is obtained in this way, that is, when $H \simeq L_\lambda$, for some Hopf algebra $L$ and partial action $\lambda \in L^*$.
In this sense, we define:

\begin{definition}\label{lda_Hopf}
	We say that a Hopf algebra $H$ is a \emph{(left) $\lambda$-Hopf algebra} if there exists a pair $(L, \lambda)$ where $L$ is a Hopf algebra and $\lambda: L \longrightarrow \k$ is a (left) partial action such that $L_\lambda \simeq H$.
\end{definition}

	Since $\varepsilon$ is a partial action (global, in fact) and $\varepsilon(h_1)h_2 = h$, for all $h \in H$, we have that $H_\varepsilon = H$. Then, every Hopf algebra $H$ is an $\varepsilon$-Hopf algebra in the usual way.

\smallbreak

In the next example, we prove that $\k G$ is a $\lambda$-Hopf algebra, for $\lambda$ other than $\varepsilon$.
\begin{exa}\label{G_eh_lda_Hopf}
	Let $G$ be a group.
	Consider a group $F$ that contains $G$ as proper subgroup and the partial action $\lambda_G : \k F \longrightarrow \k$.
	It is clear that $(\k F)_{\lambda_{G}} \simeq \k G$.
	Moreover, since $F$ can be chosen arbitrarily, we do not have uniqueness for a Hopf algebra $L$ such that $L_\lambda \simeq \k G$.
	In particular, we can consider $F$ as the group $G \times C_\ell$, where $C_\ell$ denotes the cyclic group of order $\ell$, $\ell \geq 1 $.
	In this case, $dim_\k(\k F) = \ell \ dim_\k(\k G)$.
\end{exa}

Now our goal is to determine necessary and sufficient conditions to conclude whether $H_\lambda$ is a Hopf subalgebra of $H$ or not.
First, we have the following lemma:

\begin{lema}\label{restricao}
	Let $\lambda: H \longrightarrow \k$ be a partial action of $H$ on $\k$.
	Then, $\lambda{|_{H_{\lambda}}} = \varepsilon{|_{H_{\lambda}}}$.
\end{lema}

\begin{proof}
	Consider $\lambda(h_1)h_2 \in H_\lambda$.
	Then, $\lambda(\lambda(h_1)h_2 ) = \lambda(h_1)\lambda(h_2) = \lambda(h).$
	On the other hand, $\varepsilon(\lambda(h_1)h_2 ) = \lambda(h_1)\varepsilon(h_2) = \lambda(h_1\varepsilon(h_2) ) = \lambda(h).$
Thus, $\lambda{|_{H_{\lambda}}} = \varepsilon{|_{H_{\lambda}}}$.
\end{proof}

Thus, in the subalgebra $H_\lambda$, the maps $\varepsilon$ and $\lambda$ coincide.
In particular, $\lambda$ is multiplicative on $H_\lambda$, and so if $H_\lambda$ is a subcoalgebra (or subbialgebra or Hopf subalgebra) of $H$, then the counit of $H_\lambda$ is given exactly by $\lambda{|_{H_{\lambda}}} = \varepsilon{|_{H_{\lambda}}}$.
Hence, we obtain the following characterization:

\begin{theorem}\label{carac}
	Let $H$ be a finite-dimensional Hopf algebra and $\lambda: H \longrightarrow \k$ a partial action.
	Then, $H_\lambda $ is a Hopf subalgebra of $H$ if and only if
	\begin{align}\label{eq_carac}
		\lambda(h_1)h_2 = \lambda(h_1)h_2\lambda(h_3)
	\end{align}
	holds for all $ h \in H$.
\end{theorem}

\begin{proof}
	First, suppose $\lambda(h_1)h_2 = \lambda(h_1)h_2\lambda(h_3)$ for all $h \in H$.
	We already have that $H_\lambda$ is a subalgebra of $H$.
	Consider the restriction of the comultiplication $\Delta_H$ of $H$ to $H_\lambda$, that is, $\Delta_{H}{|_{H_\lambda}} : H_\lambda \longrightarrow H \o H$.
	We denote $\Delta = \Delta_{H}{|_{H_\lambda}}$ and shall see $\Delta(H_\lambda) \subseteq H_\lambda \o H_\lambda$.
	Note that
	$$\Delta(\lambda(h_1)h_2) = \lambda(h_1)\Delta(h_2) = \lambda(h_1)h_2 \o h_3 = \lambda(h_{11})h_{12} \o h_2. $$
	Now, using the hypothesis, we get $\lambda(h_{11})h_{12} \o h_2 = \lambda(h_{11})h_{12}\lambda(h_{13}) \o h_2,$
	and so
	\begin{align*}
		\lambda(h_{11})h_{12}\lambda(h_{13}) \o h_2 & = \lambda(h_1)h_2 \lambda(h_3) \o h_4 \\
		& = \lambda(h_1)h_2 \o \lambda(h_3)h_4 \\
		& =\lambda(h_{11})h_{12} \o \lambda(h_{21}) h_{22}.
	\end{align*}
	
	Therefore, $H_\lambda$ is closed under the comultiplication $\Delta$ and, since we are dealing with restriction maps, we obtain that $(H_\lambda, \Delta, \varepsilon|{_{H_{\lambda}}})$ is a subcoalgebra of $H$.
	Moreover, $H_\lambda$ is a subbialgebra of $H$ and, since $H$ is finite dimensional, it follows that $H_\lambda$ is a Hopf subalgebra of $H$ (\emph{cf.} \cite[Proposition 7.6.1]{radford}).
	
	Conversely, suppose that $H_\lambda$ is a Hopf subalgebra of $H$, that is,
	$\Delta_{H_{\lambda}} = \Delta_{H}{|_{H_{\lambda}}}$ and $\varepsilon_{H_{\lambda}} = \varepsilon_{H}{|_{H_{\lambda}}}$.
	Then, $H_\lambda$ is a Hopf algebra and, in particular, $ \psi \circ (Id_{H_{\lambda}} \o \varepsilon_{H_{\lambda}}) \circ \Delta_{H_{\lambda}} = Id_{H_{\lambda}}$ holds,  where $\psi:H_\lambda \o \k \longrightarrow H_\lambda$ is the canonical isomorphism.
	Using Lemma \ref{restricao}, we have that
	$\lambda{|_{H_{\lambda}}} = \varepsilon_{H_{\lambda}}$ and then $ \psi \circ (Id_{H_{\lambda}} \o \lambda{|_{H_{\lambda}}}) \circ \Delta_{H_{\lambda}} = Id_{H_{\lambda}}$ holds.
	Thus, given $\lambda(h_1)h_2 \in H_\lambda$, it follows that
	\begin{align*}
		\lambda(h_1)h_2 = & Id_{H_{\lambda}}(\lambda(h_1)h_2) \\
		= & [\psi \circ (Id_{H_{\lambda}} \o \lambda{|_{H_{\lambda}}}) \circ \Delta_{H_{\lambda}}](\lambda(h_1)h_2 ) \\
		= & \lambda(h_1) h_2 \lambda(h_3).
	\end{align*}
	Therefore, $\lambda(h_1)h_2 = \lambda(h_1)h_2\lambda(h_3)$ holds for all $h \in H$.
\end{proof}

\begin{cor}\label{x_skew}
		Let $H$ be a finite-dimensional Hopf algebra and $\lambda: H \longrightarrow \k$ a partial action.
	Suppose $x \in P_{g,h}(H) \setminus \k\{g-h\}$ with $\lambda(g)=0$ and $\lambda(h)=1$.
	Then, $\lambda(x_1)x_2 \neq \lambda(x_1)x_2 \lambda(x_3)$ and, consequently, $H_\lambda$ is not a Hopf subalgebra of $H$.
\end{cor}
\begin{proof}
	Since $\Delta(x)= x \o g + h \o x$ and $\Delta_2(x) = x \o g \o g + h \o x \o g + h \o h \o x$, where $\Delta_2 = (\Delta \otimes Id_H) \circ \Delta = (Id_H \otimes \Delta) \circ \Delta$, we obtain $\lambda(x_1)x_2= \lambda(x)g + \lambda(h)x = \lambda(x)g + x$.
	On the other hand, $\lambda(x_1)x_2\lambda(x_3) = \lambda(x)g\lambda(g) +\lambda(h)x\lambda(g) + \lambda(h)h\lambda(x) = \lambda(x)h.$
	As $x \neq \lambda(x)(h-g)$, we conclude $\lambda(x_1)x_2 \neq \lambda(x_1)x_2 \lambda(x_3)$.
	
	Thus, Theorem \ref{carac} implies that $H_\lambda$ is not a Hopf subalgebra of $H$.
\end{proof}

The most important feature in the previous corollary is that $x \in H$ is a non-trivial $(g,h)$-primitive element such that $\lambda(g) \neq \lambda(h)$.
We obtain a complementary result for the case when $x$ is a non-trivial $(g,h)$-primitive element, but $\lambda(g)=1$ and $\lambda(h)=0$.

\begin{cor}\label{x_skew2}
	Let $H$ be a finite-dimensional Hopf algebra and $\lambda: H \longrightarrow \k$ a partial action.
	If $x \in P_{g,h}(H)\setminus \k \{g-h\}$ with $\lambda(g)=1$ and $\lambda(h)=0$,
	then $H_\lambda$ is not a Hopf subalgebra of $H$.
\end{cor}

\begin{proof}
	Consider $y = xh^{-1}$.
	Note that $y \in P_{gh^{-1},1}(H)\setminus \k \{gh^{-1}-1\}$ and $\lambda(gh^{-1})=0$ by Proposition \ref{N_subgrupo}.
	Hence, the result follows by Corollary \ref{x_skew}.
\end{proof}

\begin{cor}\label{nao_H_lda}
	Let $H$ be a finite-dimensional Hopf algebra and $\lambda: H \longrightarrow \k$ a partial action.
	If there exists $x \in P_{g,h}(H) \setminus \k \{g-h\}$ such that $\lambda(g) \neq \lambda(h)$, then $H_\lambda$ is not a Hopf subalgebra of $H$.
\end{cor}

\begin{proof}
	Straightforward from corollaries \ref{x_skew} and \ref{x_skew2}.
\end{proof}

\medbreak

The results above deal with $x \in P_{g,h}(H)$ such that $\lambda(g) \neq \lambda(h)$.
The next one deals when $\lambda(g) = \lambda(h)$.
\begin{cor}
	Let $H$ be a Hopf algebra and $\lambda: H \longrightarrow \k$ a partial action. Then,
	\begin{enumerate}
		\item if $g \in G(H)$ then
		$\lambda(g_1)g_2=\lambda(g_1)g_2\lambda(g_3);$
		\item if $x \in P_{g,h}(H)$ such that $\lambda(g)=\lambda(h)$, then $\lambda(x_1)x_2=\lambda(x_1)x_2\lambda(x_3).$
	\end{enumerate}
	In particular, if $H$ is finite-dimensional and has a basis given only by group-like and $(g_i,h_i)$-primitive elements, with $g_i,h_i \in G(H)$ such that $\lambda(g_i)=\lambda(h_i)$ for all $i$, then $H_\lambda$ is a Hopf subalgebra of $H$.
\end{cor}

\begin{proof}
	Recall from Proposition \ref{N_subgrupo} that $\lambda(g) \in \{0, 1_\k\}$.
	Then, item (1) is clear.
	Now, note that $\lambda(x)=0$ by Lemma \ref{propriedades_lda} (b), then
	$\lambda(x_1)x_2 = \lambda(g) x $ and $\lambda(x_1)x_2 \lambda(x_3) = \lambda(g)^2 x$. Thus, (2) holds.
	Finally, the last statement in the corollary is clear by Theorem \ref{carac}.
\end{proof}

\medbreak

Using these results, sometimes we can quickly check whether $H_\lambda$ is a Hopf subalgebra of $H$ or not.
For instance, consider $(\mathbb{H}_4)_{\lambda_\alpha}$ as in Example \ref{Sweedler's_nao_Hlda}.
Since $x \in P_{1,g}(\mathbb{H}_4)$ and $\lambda_\alpha(1) \neq \lambda_\alpha(g)$, it follows that $(\mathbb{H}_4)_{\lambda_\alpha}$ is not a Hopf subalgebra of $\mathbb{H}_4$.
We present in Subsection \ref{Subsec:exemplos} the whole setting for the partial actions computed in Subsection \ref{Subsec: 8}.

\medbreak

We exhibit now a condition stronger than \eqref{eq_carac}.
\begin{prop}\label{cond_mais_forte}
	Let $\lambda: H \longrightarrow \k$ be a partial action.
	If $\lambda(h_1)h_2=h_1 \lambda(h_2)$ for all $h \in H$, then $\lambda(h_1)h_2=\lambda(h_1) h_2 \lambda(h_3)$ for all $h \in H$.
	In this case, if $H$ is finite-dimensional, then
	$H_\lambda$ is a Hopf subalgebra of $H$.
\end{prop}

\begin{proof}
	Suppose that $\lambda(h_1)h_2=h_1 \lambda(h_2)$, for all $h \in H$.
	Then
	\begin{align*}
		\lambda(h_1) h_2 \lambda(h_3) & = \lambda(h_1) (h_2)_1 \lambda((h_2)_2) \\
		& = \lambda(h_1) \lambda((h_2)_1) (h_2)_2 \\
		&  = \lambda((h_1)_1) \lambda((h_1)_2)h_2 \\
		&  = \lambda(h_1) h_2,
	\end{align*}
	where the last equality follows from \eqref{eqn_lda}.
	Thus, if $H$ is finite-dimensional, it follows by Theorem \ref{carac} that $H_\lambda$ is a Hopf subalgebra of $H$.
\end{proof}

	The hypotheses of the previous proposition are stronger than those of Theorem \ref{carac}.
	However, checking $\lambda(h_1)h_2=h_1 \lambda(h_2)$ sometimes it is easier than checking $\lambda(h_1)h_2\lambda(h_3)=\lambda(h_1) h_2$, for $h \in H$.
	It happens because for the former equality is necessary to use $\Delta$ only once, while the latter needs $\Delta_2$.

\medbreak

Now we shall see that every Hopf algebra is a $\lambda$-Hopf algebra, for some $\lambda$ distinct from its counit.
For this purpose, we recall how partial actions of $H \o L$ on $A \o B$ are canonically obtained from partial actions of $H$ and $L$ on $A$ and $B$, respectively.

\begin{rem}\label{produto_acoes}
	Let $H$ and $L$ be two Hopf algebras. Consider $\cdot_H$ and $\cdot_L$ partial actions of $H$ and $L$ on algebras $A$ and $B$, respectively.
	Then, the linear map $\rightharpoonup = (\cdot_H \o \cdot_L) \circ (Id_{H} \o \tau_{L,A} \o Id_B )$ is a partial action of $H \o L$ on $A\o B$, where $\tau_{L,A} : L \o A \longrightarrow A \o L $ is the canonical twist isomorphism between $L$ and $A$.
	Moreover, if $\cdot_H$ and $\cdot_L$ are symmetric, then so is $\rightharpoonup$.
	In particular, since $\k \simeq \k \o \k$, we obtain partial actions of $H \o L$ on $\k$ from partial actions of $H$ and $L$ on $\k$.
\end{rem}

We emphasize that the partial actions of $H \o L$ on $A \o B$, constructed as above, in general do not cover all partial actions of $H \o L$ on $A \o B$, even when $A = B = \k$.
As a very simple example, take $A=B=\k$ and $H = L = \k C_2$.
Then, $H \o L = \k C_2 \o \k C_2 \simeq \k( C_2 \times C_2)$.
Recall, by Example \ref{lda_kg}, that a partial action of $\k (C_2 \times C_2)$ on $\k$ is parametrized by a subgroup of $C_2 \times C_2$.
Since for $C_2 = \{1, g \}$ we have only the trivial subgroups, we only have two partial actions of $\k C_2$ on $\k$, namely, $\lambda_{\{1,g\}} = \varepsilon_{\k C_2}$ and $\lambda_{\{1\}}$.
Hence, by Remark \ref{produto_acoes}, we obtain only $4$ partial actions of $\k C_2 \o \k C_2 \simeq \k( C_2 \times C_2)$ on $\k$: $\lambda_{ \{  1 \}} \o \lambda_{\{  1 \}}=\lambda_{ \{  (1, 1 ) \}} $, $\lambda_{ \{  1 \}} \o \varepsilon_{\k C_2} =\lambda_{ \{  (1, 1), ( 1, g) \}} $, $\varepsilon_{\k C_2} \o \lambda_{\{  1 \}}=\lambda_{ \{  (1, 1), (g, 1) \}} $ and $\varepsilon_{\k C_2} \o \varepsilon_{\k C_2} = \varepsilon_{\k (C_2 \times C_2)}.$
However, since $C_2 \times C_2$ has $5$ subgroups, there is a partial action that is not covered by this construction.
Precisely the partial action $\lambda_{ \{  (1, 1), (g, g) \}}.$

This fact can also be observed in the examples given in Subsection \ref{Subsec: 8}.
For instance, we obtain 4 partial actions of $\mathbf{H_{12}}$ on $\k$ using the corresponding partial actions of $\mathcal{A}_{2}$ and $\k C_2$ on $\k$.
Namely, $\varepsilon_{\mathcal{A}_{2}} \o \varepsilon_{\k C_2} = \varepsilon_{\mathbf{H_{12}}}$, $\varepsilon_{\mathcal{A}_{2}} \o \lambda_{\{1\}}^{\k C_{2}} = \lambda_{\{1, g\}}$, $\lambda_{\{1\}}^{\mathcal{A}_{2}} \o \varepsilon_{\k C_2} = \lambda_{\{1, h\}}$ and $\lambda_{\{1\}}^{\mathcal{A}_{2}} \o \lambda_{\{1\}}^{\k C_{2}}  = \lambda_{\{1\}}$.
However, note that the partial action $\lambda_{\{1, gh\}}$ is not obtained using Remark \ref{produto_acoes}.

\medbreak

We use Remark \ref{produto_acoes} to conclude that every Hopf algebra is a $\lambda$-Hopf algebra, where $\lambda$ is a genuine partial action, that is, distinct from the counit, as follows.

\begin{theorem}\label{toda_H_eh_lda}
	Let $H$ be a Hopf algebra.
	Then, there exists a Hopf algebra $L$ and a partial action $\lambda: L \longrightarrow \k$ such that $L_\lambda \simeq H$.
\end{theorem}

\begin{proof}
	Let $G$ be group, $G \neq \{1_G\}$, where $1_G$ denotes the identity element of $G$.
	Consider the partial action of $\k G$ on $\k$ parametrized by the subgroup $\{1_G\}$, that is, $\lambda_{ \{1_G\}} : \k G \longrightarrow \k$ is given by $\lambda_{\{1_G\}}(g) = \delta_{1_G,g}$, that is, $\lambda_{\{1_G\}}(1_G) = 1_\k$ and $\lambda_{\{1_G\}}(g) =0$, for all $g \in G$, $g \neq 1_G$.
	By Remark \ref{produto_acoes}, $\lambda = \varepsilon_H \o \lambda_{\{1_G\}}$ is a partial action of $H \o \k G$ on $\k$.
	Moreover, an easy calculation shows that $(H \o \k G)_{\lambda} = H \o 1_{\k G} \simeq H$.
	Indeed, for $h \o g \in H \o \k G$, we have
	\begin{align*}
		\lambda(( h \o g)_1) (h \o g)_2  = \varepsilon(h_1)\lambda_{\{1_G\}}(g) (h_2 \o g) = h \o \lambda_{\{1_G\}}(g)g  = h \o \delta_{1_G, g} g.
	\end{align*}
	Then, for $L= H \o \k G$ and $\lambda$ as above, it follows that
	$$L_\lambda = \langle \lambda  ( (h \o g)_1 ) (h \o g)_2  \ | \ h \in H, g \in G \rangle =\langle h \o 1_G \ | \ h \in H \rangle \simeq H,$$
	as Hopf algebras.
Hence, $H$ is a $\lambda$-Hopf algebra.	
\end{proof}

\medbreak

The result above is constructive, however, this is not the only way to obtain a Hopf algebra $L$ and a partial action of $L$ on $\k$ such that $L_\lambda \simeq H$.
We shall see in the next subsection that
Taft's algebra is a $\lambda$-Hopf algebra by another way.
Furthermore, since the group $G$ can be chosen arbitrarily in the proof of Theorem \ref{toda_H_eh_lda}, we obtain that $H$ can be embedded into the Hopf algebra $L$, where $dim_\k(L)$ is quite arbitrary, and there exists a partial action $\lambda: L \longrightarrow \k$ such that $L_\lambda \simeq H$.
See Example \ref{G_eh_lda_Hopf}.

\medbreak

To end this subsection, we characterize $\lambda$-Hopf algebras as partial matched pairs of a Hopf algebra and its base field, and then we use this result to produce examples of partial matched pairs.
After, in  Subsection \ref{Subsec:exemplos}, we present the complete classification of such partial matched pairs for pointed non-semisimple Hopf algebras $H$ with $\dim_\k(H)=8,16.$
 See \cite{matchedpair} for details of the theory of partial matched pairs.
\begin{prop}\label{lda_Hopf_e_par_comb}
Let $L$ be a Hopf algebra. Then, $(L,\k)$ is a partial matched pair if and only if $L_\lambda$ is a $\lambda$-Hopf algebra.
\end{prop}
\begin{proof}
	It follows from Theorem \ref{carac} and \cite[Proposition 4.4]{matchedpair}.
\end{proof}

\begin{cor}
		Let $H$ be a Hopf algebra. Then, there exists a Hopf algebra $L$ such that $H$ is the partial matched pair $(L,\k)$. 
\end{cor}
\begin{proof}
	Immediately from Theorem \ref{toda_H_eh_lda} and Proposition \ref{lda_Hopf_e_par_comb}.
\end{proof}

\subsection{Taft's algebra as a $\lambda$-Hopf algebra} \label{Sec: Taft como lambda Hopf}
Throughout this subsection, we assume that $\k$ is an algebraically closed field of characteristic zero.

\medbreak

The Taft's algebra of order $n$, here denoted by $T_n(q)$, is an important Hopf algebra presented by Taft in \cite{Taft}.
We briefly recall its definition.
Let $n \geq 2$ be a positive integer and $q$ a primitive root of unity of order $n$.
As algebra
$T_n(q) = \langle g, \; x,\;  | \; g^{n}=1, \; x^n=0, \; xg = q gx \rangle.$
Thus, $\mathcal{B} = \{g^ix^j  \, | \,  0 \leq i,j \leq n-1  \}$ is the canonical basis for $T_n(q)$, and consequently $dim_\k(T_n(q)) = n^{2}$.
For the coalgebra structure, $g$ is a group-like element and $x$ is a $(1,g)$-primitive element.
In general, we have $\Delta (g^ix^j) = \sum_{\ell=0}^j {j \choose \ell}_{q}  g^{i+\ell}x^{j-\ell} \o g^ix^\ell$.
So, completing the Hopf structure of $T_n(q)$, set $S(g) =g^{n-1}$ and $S(x) = -g^{n-1}x.$
Notice that  $G( T_n(q) ) = \langle g \rangle = \{ 1, g, \cdots, g^{n-1} \} =  C_{n}.$

\medbreak

In the sequel, we present $T_n(q)$ as a $\lambda$-Hopf algebra, where the Hopf algebra $L$ such that $L_\lambda \simeq T_n(q)$ is not simply $T_n(q)$ tensorized by a group algebra. See Theorem \ref{toda_H_eh_lda}.
Aiming this, we recall a family of Hopf algebras, here denoted by $T_n^k(\omega)$, and calculate a suitable partial action $\lambda_k: T_n^k(\omega) \longrightarrow \k$ such that $(T_n^k(\omega))_{\lambda_k} \simeq T_n(\omega^k)$.

\medbreak

The Hopf algebra $T_n^k(\omega)$ is a generalization of the Taft's algebra $T_n(q)$ (see \cite{nicolas_p3} and \cite[Appendix]{counting_arguments}).
Let $k$ and $n$ be positive integers, $n \geq 2$, and $\omega$ a primitive root of unity of order $kn$.
As algebra, $T_n^k(\omega)$ is generated by the letters $g$ and $x$ with the relations $g^{kn}=1,$ $x^n =0$ and $xg = \omega gx$.
Thus, $\mathcal{B} = \{g^ix^j  \, | \,  0 \leq j \leq n-1, \, 0 \leq i \leq kn-1  \}$ is the canonical basis of $T_n^k(\omega)$, and consequently $dim_\k(T_n^k(\omega)) = kn^{2}$.
For the coalgebra structure, $g$ is a group-like element and $x$ is a $(1,g^k)$-primitive element.
In general, $\Delta (g^ix^j) = \sum_{\ell=0}^j {j \choose \ell}_{\omega^k}  g^{i+\ell k}x^{j-\ell} \o g^ix^\ell$.
To complete the Hopf structure of $T_n^k(\omega)$, set $S(g) =g^{kn-1}$ and $S(x) = -g^{kn-k}x.$
Also, note that $G( T_n^k(\omega) ) = \langle g \rangle = \{ 1, g, \cdots, g^{kn-1} \} =  C_{kn}.$

\medbreak

\begin{rem}\label{taft_subalg_taft_general}
	Notice that $T_n^k(\omega)$ generalizes Taft's algebras, since when $k=1$, $T_n^1(\omega)$ is exactly the Taft's algebra of order $n$.
	Moreover, for any positive integer $k$, the Taft's algebra $T_n(q)$ is embedded as a Hopf subalgebra of $T_n^k(\omega)$, where $q=\omega^k$.
	Indeed, consider $T_n(q) = \langle h, \;  y, \; |  \; h^{n}=1, \; y^n=0, \; yh = q hy \rangle,$
	where $q=\omega^k$.
	Then, $\varphi : T_n(q) \longrightarrow T_n^k(\omega)$ given by $\varphi(y)=x$ and $\varphi(h)=g^k$ is an injective homomorphism of Hopf algebras.
\end{rem}

\medbreak

Now, we want to calculate a suitable partial action of $T_n^k(\omega)$ on $\k$.
Consider the subgroup $N= \langle g^k \rangle$ of $G(T_n^k(\omega))$ and the linear map $\lambda_k : T_n^k(\omega) \longrightarrow \k$ given by $\lambda(g^ix^j)=\delta_{j,0}\delta_N(g^i)$ for all $i,j \in \Z$, $j \geq 0$, where the symbol $\delta_{j,0}$ stands for the Kronecker delta and $\delta_N$ defined as in Example \ref{lda_kg}.

We use Proposition \ref{k_mod_alg_parc} to prove that $\lambda_k$ as defined above is a partial action of $T_n^k(\omega)$ on $\k$.
Denotes $\lambda_k$ simply by $\lambda$.
Clearly $\lambda(1)=1_\k$.
Now, it remains only to verify condition \eqref{eqn_lda}, that is, if
$$\lambda(u)\lambda(v)=\lambda(u_1)\lambda(u_2v)$$
holds for all $u, v \in T_n^k(\omega)$.

First, if $u=g^i$, $i \in \Z$, then \eqref{eqn_lda} is satisfied for all $v \in T_n^k(\omega)$.
Indeed, if $v= g^rx^s$, with $r,s \in \Z$, $s\geq 0$, then
$$\lambda(u)\lambda(v) = \lambda(g^i)\lambda(g^rx^s) = \delta_{0,0}\delta_N(g^i)\delta_{s,0}\delta_N(g^r) = \delta_N(g^i)\delta_{s,0}\delta_N(g^r).$$
On the other hand,
\begin{align*}
	\lambda(u_1)\lambda(u_2v) & = \lambda(g^i)\lambda(g^ig^rx^s) = \lambda(g^i)\lambda(g^{i+r}x^s) \\
	& = \delta_{0,0}\delta_N(g^i)\delta_{s,0}\delta_N(g^{i+r}) = \delta_N(g^i)\delta_{s,0}\delta_N(g^{i+r}).
\end{align*}

Thus, if $g^i \notin N$, then $\lambda(g^i)\lambda(g^{i+r}x^s)=0$ and $\delta_N(g^i)\delta_{s,0}\delta_N(g^{i+r})=0$, since $\delta_N(g^i)=0$.
Otherwise, if $g^i \in N$, then $\delta_N(g^r) = \delta_N(g^{i+r})$.
Hence, the equality desired holds.

Now, assume $u=g^ix^j$ with $i,j \in \Z$, $j \geq 1$.
Then, for all $v \in T_n^k(\omega)$, \eqref{eqn_lda} means
$$\lambda(g^ix^j)\lambda(v) = \sum_{\ell=0}^j {j \choose \ell}_{\omega^k}  \lambda(g^{i+\ell k}x^{j-\ell}) \lambda(g^ix^\ell v).$$

Since $j \geq 1$, we get $\lambda(g^ix^j)=\delta_{j,0}\delta_N(g^i) = 0$ and $\lambda(g^{i+\ell k}x^{j-\ell}) = \delta_{j-\ell,0}\delta_N(g^{i+\ell k}) = 0,$ for all $0 \leq \ell \leq (j-1)$.
Then, the previous sum is reduced only when $\ell=j$, and results in
$$0 = {j \choose j}_{\omega^k}  \lambda(g^{i+jk}x^{j-j}) \lambda(g^ix^jv)=\lambda(g^{i+jk}) \lambda(g^ix^jv)$$
for all $v \in T_n^k(\omega)$.

For $v=g^sx^t$ with $s,t \in \Z$, $t \geq 0$, it follows that $j+t \geq 1$ and consequently $\lambda(g^{i+s}x^{j+t})= \delta_{j+t,0}\delta_N(g^{i+s}) = 0$.
Hence,
\begin{align*}
	\lambda(g^{i+jk}) \lambda(g^ix^jv) & = \lambda(g^{i+jk}) \lambda(g^ix^jg^sx^t) = \omega^{js}\lambda(g^{i+jk}) \lambda(g^{i+s}x^{j+t}) =0.
\end{align*}

Thus, \eqref{eqn_lda} holds for all $u,v \in T_n^k(\omega)$, and so $\lambda$ is a partial action of $T_n^k(\omega)$ on $\k$.

Finally, it remains to check that $(T_n^k(\omega))_\lambda = T_n(\omega^k)$.
Recall that $(T_n^k(\omega))_\lambda = \{ \lambda(u_1)u_2 \ | \ u \in T_n^k(\omega)\}$.
Thus, for $u = g^ix^j$, $i,j \in \Z$, $j \geq 0$, we have
$$\lambda((g^ix^j)_1)(g^ix^j)_2 = \sum_{\ell=0}^j {j \choose \ell}_{\omega^k}  \lambda(g^{i+\ell k}x^{j-\ell}) g^ix^\ell.$$

Since $\lambda(g^{i+\ell k}x^{j-\ell}) = \delta_{j-\ell,0}\delta_N(g^{i+\ell k}) = 0$, for any $0 \leq \ell \leq (j-1)$, we conclude that
$\lambda((g^ix^j)_1)(g^ix^j)_2 = {j \choose j}_{\omega^k}  \lambda(g^{i+jk}x^0) g^ix^j=\lambda(g^{i+jk}) g^ix^j = \delta_N(g^{i+jk})g^ix^j,$
for all $i,j \in \Z$, $j \geq 0$.

As $g^{jk} \in N$, it follows that $\delta_N(g^{i+jk}) =\delta_N(g^{i})$ and consequently
$\lambda((g^ix^j)_1)(g^ix^j)_2 =  \delta_N(g^{i}) g^ix^j$.
Therefore,
\begin{align*}
	(T_n^k(\omega))_\lambda & =\{ \lambda(u_1)u_2 \ | \ u \in T_n^k(\omega)\} \\
	& = \langle \lambda(u_1)u_2 \ | \ u \in \mathcal{B} \rangle \\
	& = \langle \lambda((g^ix^j)_1)(g^ix^j)_2 \ | \ 0 \leq  i \leq (kn-1), 0 \leq j \leq (n-1) \rangle \\
	& = \langle \delta_N(g^{i}) g^ix^j \ | \ 0 \leq  i \leq (kn-1), 0 \leq j \leq (n-1) \rangle \\
	& = \langle \delta_N(g^{rk+s}) g^{rk+s}x^j \ | \ 0 \leq s \leq (k-1), 0 \leq r, j \leq (n-1) \rangle \\
	& = \langle \delta_N(g^{rk}) g^{rk}x^j \ | \ 0 \leq  r,j \leq (n-1) \rangle \\
	& = \langle  g^{rk}x^j \ | \ 0 \leq  r, j \leq (n-1) \rangle \\
	& = \varphi(T_n(\omega^k)) \\
	& \simeq T_n(\omega^k),
\end{align*}
where $\varphi$ is the injective homomorphism of Hopf algebras given in Remark \ref{taft_subalg_taft_general}.

\medbreak

	We observe that the choice of the parameter $k$ was arbitrary.
	Thus, the dimension of $T_n^k(\omega)$ can be chosen as large as desired.
	This situation is the same as for group algebras (Example \ref{grupo_eh_Hlda}) and also in Theorem \ref{toda_H_eh_lda}.
	
	\medbreak
	
\begin{exa}
	Consider the Sweedler's Hopf algebra $\mathbb{H}_4 = T_2(-1)$.
	Note that $T_2^2(q)= \mathcal{A}^{'''}_{4,q}$ and $T_2^4(\tilde{q})=\mathbf{H_{19}}$, where $q$ and $\tilde{q}$ are primitive $4^{th}, 8^{th}$ roots of unity, respectively (see Subsection \ref{Subsec: 8}).
	Then, the partial actions $\lambda_2 = \lambda_{\{1, g^2\}}$ of $\mathcal{A}^{'''}_{4,q}$ on $\k$ and $\lambda_4= \lambda_{\{1, g^4\}}$  of $\mathbf{H_{19}}$ on $\k$, are such that $(\mathcal{A}^{'''}_{4,q})_{\lambda_2} \simeq \mathbb{H}_4 \simeq (\mathbf{H_{19}})_{\lambda_4}$.
\end{exa}

\subsection{Examples of $\lambda$-Hopf algebras} \label{Subsec:exemplos}
In this subsection, we show, schematically, the $\lambda$-Hopf algebras obtained with the partial actions of each pointed non-semisimple Hopf algebra $H$ with $dim_\k(H)=8,16$ (see Subsection \ref{Subsec: 8}).
Thus, by Proposition \ref{lda_Hopf_e_par_comb}, we obtain all partial matched pairs $(H, \k)$, for each one of these Hopf algebras.

\medbreak

\begin{theorem} Let $H$ be one of the following pointed non-semisimple Hopf algebras:
	\begin{itemize}
		\item \underline{Dimension 8:} $\mathcal{A}_2, \mathcal{A}^{'}_4$ or $\mathcal{A}^{''}_4$;
		\item \underline{Dimension 16:} $\mathbf{H_{1}}$, $\mathbf{H_{4}},$ $\mathbf{H_{5}},$ $\mathbf{H_{6}},$ $\mathbf{H_{7}},$ $ \mathbf{H_{8}},$ $\mathbf{H_{9}},$ $\mathbf{H_{10}},$ $\mathbf{H_{11}},$ $\mathbf{H_{14}},$ $\mathbf{H_{15}},$ $\mathbf{H_{16}},$ $\mathbf{H_{18}}$ or $\mathbf{H_{22}}$.
	\end{itemize}
	Then, $H$ has not a partial action $\lambda$, $\lambda \neq \varepsilon$, such that $H_\lambda$ is a Hopf subalgebra of $H$.
\end{theorem}

\begin{proof}
	Note that for these Hopf algebras, there exists a $(1,g)$-primitive element $x$.
	Almost all partial actions $\lambda (\neq \varepsilon)$ of these Hopf algebras are such that $\lambda(g)=0$.
	The only exceptions are $\lambda_{\{1, g\}}$ of $\mathbf{H_{i}}$, $i \in \{14, 15\}$.
	For these two specific cases, there exists $y \in \mathbf{H_i}$ such that $y$ is a $(1,h)$-primitive element and $\lambda_{\{1, g\}}(h)=0$.
	Thus, by Corollary \ref{nao_H_lda}, it follows that $H_\lambda$ is not a Hopf algebra of $H$.
\end{proof}

\medbreak

In the following diagram, an arrow from $H$ to $B$, $H \stackrel{\lambda}{\longrightarrow}B$, means that $\lambda$ is a partial action of $H$ on $\k$ such that $B \simeq H_\lambda$, as Hopf algebras.
At this point, Corollary \ref{nao_H_lda} is useful to discard a lot of partial actions $\lambda$ of $H$ on $\k$, since the existence of a $(g,h)$-primitive element such that $\lambda(g) \neq \lambda(h)$ implies that $H_\lambda$ is not a Hopf subalgebra of $H$.
To verify whether $H_\lambda$ is or not a Hopf subalgebra of $H$, one uses Theorem \ref{carac} or, which is usually easier, Proposition \ref{cond_mais_forte}.

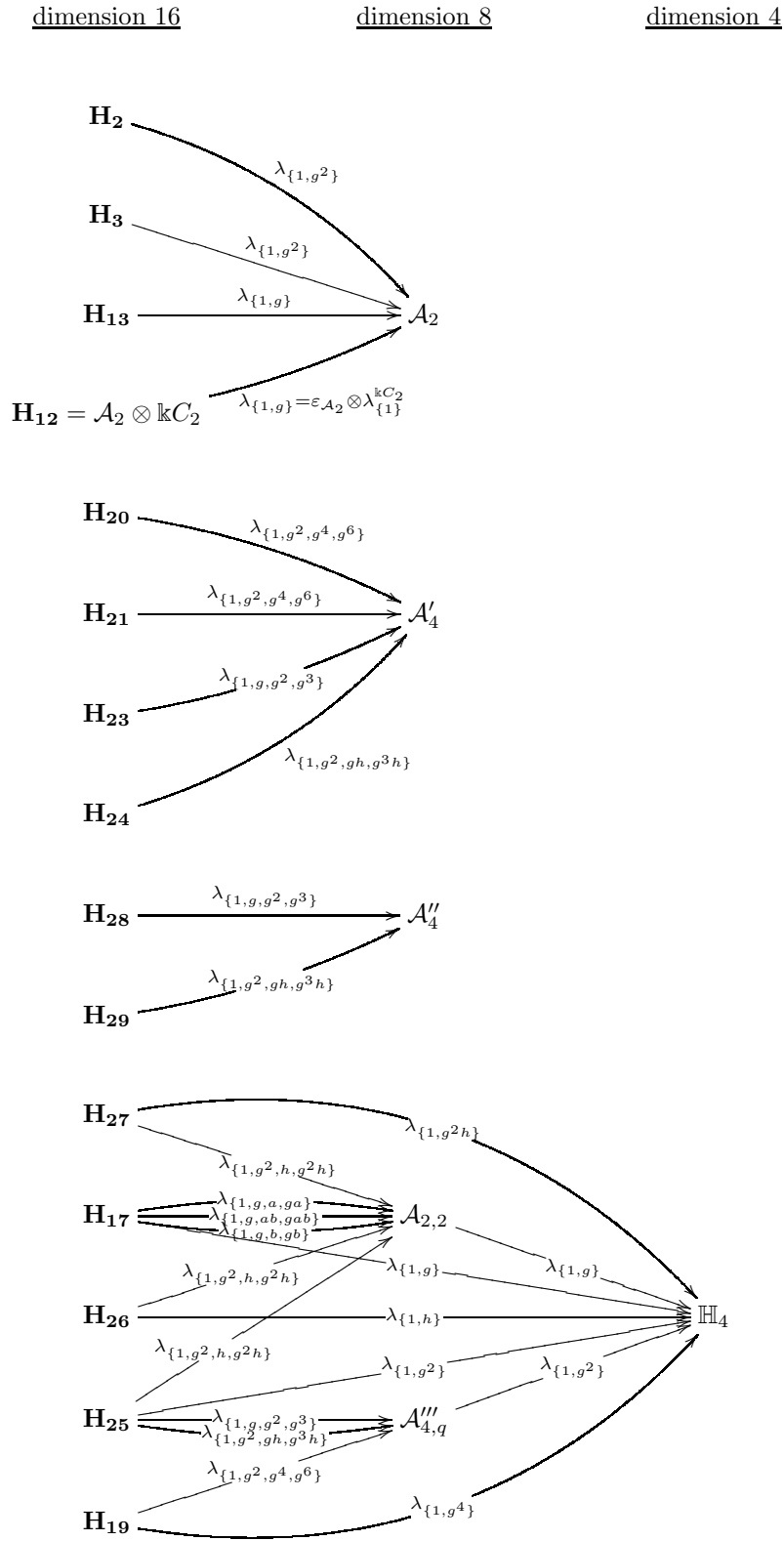
$$
\xymatrix{
	\underline{\textrm{dimension} \, \, 16}
	& &	\underline{\textrm{dimension} \, \, 8} & &	\underline{\textrm{dimension} \, \, 4}
	\\
	\mathbf{H_{2}} \ar@/^1pc/[ddrr]^{\lambda_{\{1,g^2\}}} & &
	\\
	\mathbf{H_{3}} \ar[drr]^{\lambda_{\{1,g^2\}}} & &
	\\
	\mathbf{H_{13}} \ar[rr]^{\lambda_{\{1,g\}}}  & & \mathcal{A}_{2} & &
	\\
	\mathbf{H_{12}} = \mathcal{A}_2 \o \k C_2 \ar@/_/[urr]_{\,\,\, \quad\lambda_{\{1,g\}} = \varepsilon_{\mathcal{A}_2} \o \lambda_{\{1\}}^{\k C_2}}  & &
%}
%$$
%
%$$
%\xymatrix{
%	\underline{\textrm{dimension} \, \, 16}
%	& &	\underline{\textrm{dimension} \, \, 8} & &	\underline{\textrm{dimension} \, \, 4}
	\\
	\mathbf{H_{20}} \ar@/^/[drr]^{\quad \lambda_{\{1,g^2,g^4,g^6\}}} & & & &
	\\
	\mathbf{H_{21}} \ar[rr]^{\lambda_{\{1,g^2,g^4,g^6\}}}  & & \mathcal{A}_{4}'  & &
	\\
	\mathbf{H_{23}} \ar@/_/[urr]|-{\lambda_{\{1,g,g^2,g^3\}}} & & & &
	\\
	\mathbf{H_{24}} \ar@/_1pc/[uurr]_{\quad \lambda_{\{1,g^2,gh,g^3h\}}} & &  & &
		\\
	\mathbf{H_{28}} \ar[rr]^{\lambda_{\{1,g,g^2,g^3\}}} & & \mathcal{A}_{4}''  & &
	\\
	\mathbf{H_{29}}  \ar@/_/[urr]|-{\lambda_{\{1,g^2,gh,g^3h\}}} & &
	\\
	\mathbf{H_{27}} \ar[drr]|-{\quad \lambda_{\{1,g^2,h, g^2h\}}} \ar@/^3pc/[ddrrrr]|-{\lambda_{\{1,g^2h\}}}  & &  & &
	\\
	\mathbf{H_{17}} \ar@/^/[rr]|-{\lambda_{\{1, g, a, ga\}}} \ar[rr]|-{\lambda_{\{1, g, ab, gab\}}}  \ar@/_/[rr]|-{\lambda_{\{1, g, b, gb\}}}  \ar[rrrrd]|-{\lambda_{\{1,g\}}}  & &  \mathcal{A}_{2,2} \ar[drr]|-{\lambda_{\{1,g\}}}  & &
	\\
	\mathbf{H_{26}} \ar[urr]|-<<<<<<<<<<<<<<{\lambda_{\{1,g^2,h, g^2h\}}} \ar[rrrr]|-{\lambda_{\{1, h\}}}   & &  & & \mathbb{H}_4
	\\
	\mathbf{H_{25}}  \ar[uurr]|-<<<<<<<<<<<<{\lambda_{\{1,g^2,h,g^2h\}}} \ar[rr]|-{\lambda_{\{1,g,g^2,g^3\}}} \ar[urrrr]|-{\lambda_{\{1,g^2\}}} \ar@/_/[rr]|-{\lambda_{\{1,g^2,gh,g^3h\}}} & & \mathcal{A}_{4,q}'''  \ar[urr]|-{\lambda_{\{1,g^2\}}} & &
	\\
	\mathbf{H_{19}}   \ar[urr]|-{\lambda_{\{1,g^2,g^4, g^6\}}} \ar@/_3pc/[uurrrr]|-{\lambda_{\{1,g^4\}}} & & & &
}
$$

\subsection*{Acknowledgments} 
The authors would like to thank the referee for the corrections and suggestions provided.

\bibliographystyle{abbrv}

\end{document}